\documentclass[amscd,amssymb,verbatim,10pt]{amsart}
\usepackage[fontsize = 10bp]{fontsize}
\usepackage{scalerel}
\usepackage{stackengine,wasysym}
\usepackage{bbm}
\usepackage{enumerate}
\usepackage{mathtools}

\usepackage{amssymb,latexsym, amsmath, amscd, array, graphicx, pb-diagram}
\usepackage{hyperref}
\usepackage{color}

\usepackage{tikz-cd}

\usepackage{color}
\hypersetup{
    colorlinks=true,
    linkcolor=red,
    urlcolor=blue,
    citecolor=blue
           }

\usepackage{amsthm}
\usepackage{hyperref}
\usepackage[all]{xy}
\usepackage[scr]{rsfso}

\numberwithin{equation}{section}

\swapnumbers

\theoremstyle{plain}

\newtheorem{thm}{Theorem}[section]

\newtheorem{lemma}[thm]{Lemma}

\newtheorem{proposition}[thm]{Proposition}
\newtheorem{corollary}[thm]{Corollary}

\theoremstyle{definition}

\newtheorem{definition}[thm]{Definition}

\newtheorem{remark}[thm]{Remark}

\newtheorem{ex}[thm]{Example}

\DeclareMathOperator{\cat}{{\mathsf{cat}}}
\DeclareMathOperator{\TC}{{\mathsf{TC}}}
\DeclareMathOperator{\zcl}{{\mathsf{zcl}}}
\DeclareMathOperator{\gap}{{\mathsf{gap}}}

\DeclareMathOperator{\proj}{{\textup{proj}}}

\DeclareMathOperator{\Ker}{{\rm Ker}}

\newcommand{\D}{\mathscr{D}}

\def\Im{\protect\operatorname{Im}}

\makeatletter
\def\@tocline#1#2#3#4#5#6#7{\relax
  \ifnum #1>\c@tocdepth 
  \else
    \par \addpenalty\@secpenalty\addvspace{#2}%
    \begingroup \hyphenpenalty\@M
    \@ifempty{#4}{%
      \@tempdima\csname r@tocindent\number#1\endcsname\relax
    }{%
      \@tempdima#4\relax
    }%
    \parindent\z@ \leftskip#3\relax \advance\leftskip\@tempdima\relax
    \rightskip\@pnumwidth plus4em \parfillskip-\@pnumwidth
    #5\leavevmode\hskip-\@tempdima
      \ifcase #1
       \or\or \hskip 1em \or \hskip 2em \else \hskip 3em \fi%
      #6\nobreak\relax
    \hfill\hbox to\@pnumwidth{\@tocpagenum{#7}}\par
    \nobreak
    \endgroup
  \fi}
\makeatother

\def\C{{\mathbb C}}
\def\Z{{\mathbb Z}}
\def\Q{{\mathbb Q}}
\def\R{{\mathbb R}}

\def\1{\hbox{\rm\rlap {1}\hskip.03in{\rom I}}}
\def\Bbbone{{\rm1\mathchoice{\kern-0.25em}{\kern-0.25em}
{\kern-0.2em}{\kern-0.2em}I}}

\usepackage{comment}

\long\def\forget#1\forgotten{} %

\newcommand\ver[1]{\marginpar{\tiny Changed in Ver \VER}}

\date{\today}

\usepackage{adjustbox}

\newcommand{\vgeq}{\mathrel{\rotatebox{90}{$\geq$}}}

\begin{document}

\begin{abstract}
For positive integers $k$, $n$, and $g$ with $k\geq2$, we give a closed-form expression for the $k$-th $\mathbb{Z}_2$-zero-divisor cup length $\zcl_k(SP^n(N_g))$ of the $n$-th symmetric product $SP^n(N_g)$ of the closed non-orientable surface $N_g$ of genus $g$. This allows us to estimate, and in some cases, completely determine, the $k$-th sequential topological complexity $\TC_k(SP^n(N_g))$, as well as the Lusternik--Schnirelmann category of the homotopy cofiber of the $k$-th diagonal map $SP^n(N_g) \to (SP^n(N_g))^k$. Our results recover previously known facts for even-dimensional real projective spaces ($g=1$) and closed non-orientable surfaces ($n=1$). In addition, we show that, as $g$ grows, $\TC_2(SP^n(N_g))$ behaves in a different way as all other invariants $\TC_k(SP^n(N_g))$ do. Likewise, as $k$ grows, we describe an eventual maximal-possible linear growth of $\zcl_k(SP^n(N_g))$, which allows us to prove the rationality conjecture of Farber and Oprea for the TC-generating function of $SP^n(N_g)$.
\end{abstract}

\title[Symmetric products of non-orientable surfaces]{On the sequential topological complexity and the LS-category of the cofiber of higher diagonals for symmetric products of non-orientable surfaces}

\author[J.~Gonz\'alez]{Jes\'us~Gonz\'alez}

\author[E.~Jauhari]{Ekansh~Jauhari}

\address{Jes\'us~Gonz\'alez, Departamento de Matem\'aticas, Centro de Investigaci\' on y de Estudios Avanzados del IPN, Av. Instituto Polit\'ecnico Nacional 2508, San Pedro Zacatenco, Ciudad de M\'exico, 07000.}

\email{jesus.glz-espino@cinvestav.mx and jesus@math.cinvestav.mx}

\address{Ekansh Jauhari, Department of Mathematics, University of Florida, 358 Little Hall, Gainesville, FL 32611, USA.}

\email{ekanshjauhari@ufl.edu}



\maketitle
\tableofcontents

\section{Introduction and main results}\label{sec: intro}
The $k$-th sequential topological complexity $\TC_k(X)$ of a topological space $X$ is a numerical homotopy invariant, central to the field of topological robotics, that measures the difficulty in moving points continuously within the space $X$,~\cite{Far,Ru}. More precisely, $\TC_k(X)$ is one less than the minimal number of \emph{motion planning rules} required to tell an autonomous system (such as a robot) how to move between any $k$ points of $X$. As is evident from the vast literature on topological complexity, computing $\TC_k(X)$ for a general space $X$ is a difficult task; so typically, one studies lower bounds and upper bounds to $\TC_k(X)$ that often come from the cohomology of $X^k$ and the homotopy dimension of $X$, respectively. One of the most useful lower bounds to $\TC_k(X)$ is the $k$-th zero-divisor cup length of $X$, denoted $\zcl_k(X)$, which is the cup length of the kernel of the homomorphism induced in cohomology by the $k$-th diagonal map $X\to X^k$. Indeed, this invariant determines $\TC_k(X)$ for $k\ge 2$ when $X$ is a finite product of spheres~\cite{Far,BGRT}, a simply connected symplectic manifold~\cite{FTY,BGRT}, an orientable surface~\cite{Far,Ru} or its finite symmetric product~\cite{DCDJ,Ja}, or a non-orientable surface for $k\ge 3$~\cite{GGGL}, just to name a few classes of closed manifolds.

The classical \emph{motion planning problem} of continuously moving between pairs of points on a real projective space $P^n$ dates back to the work of~\cite{FTY}, where it was proved that $\TC_2(P^n)$ coincides with the immersion dimension of $P^n$ when $n\ne 1,3,7$, and with $n$ otherwise. More recently, the \emph{sequential} motion planning problem of moving between the coordinates of $k$-tuples of points on $P^n$ was studied in~\cite{CAG+,Dav}, where sharp lower bounds to $\TC_k(P^n)$ were obtained in terms of the $k$-th $\Z_2$-zero-divisor cup length of $P^n$. In this paper, we extend the work on even-dimensional real projective spaces $P^{2n}$ from~\cite{Dav} to the $n$-th symmetric product $SP^n(N_g)$ of a closed non-orientable surface $N_g$ of genus $g\ge 1$ for $n\ge 1$. Namely, we estimate (and in some cases, precisely determine) the value $\TC_k(SP^n(N_g))$ for each $k\ge 2$ by giving closed-form expressions for the $k$-th $\Z_2$-zero-divisor cup length of any $SP^n(N_g)$. This work is an extension because, as noted by Arnold, Maxwell's theorem on multipole representations of spherical harmonics yields a homeomorphism $SP^n(P^2)\cong P^{2n}$, which can also be interpreted as a ``quaternionic'' analogue of the well-known homeomorphism $SP^n(\C P^1)\cong \C P^n$ implied by the Fundamental Theorem of Algebra, see~\cite{BGZ,KS}.

Another useful lower bound to $\TC_k(X)$ is given in terms of the Lusternik--Schnirelmann category of the homotopy cofiber of the $k$-th diagonal map $X\to X^k$. This invariant, which we denote by $\cat(C_{\Delta_kX})$, was studied in detail for $k=2$ in~\cite{GV,Dr1,Dr2,GGL}, for example, where it played a central role in the study of the classical topological complexity of closed non-orientable surfaces. For general $k$, a close relationship between $\TC_k(X)$ and $\cat(C_{\Delta_kX})$ has recently been settled in~\cite[Proposition~2.1~(2)]{CAC}. In this paper, we establish general comparisons between $\zcl_k(X)$, $\cat(C_{\Delta_kX})$, and $\TC_k(X)$, which extend known properties for $k=2$ and will play a central role in our assessment of $\cat(C_{\Delta_kSP^n(N_g)})$ and $\TC_k(SP^n(N_g))$. In particular, $\cat(C_{\Delta_kX})$ is bounded from below by $\zcl_k(X)$, so our computations of the $k$-th $\Z_2$-zero-divisor cup length of $SP^n(N_g)$ in this paper give estimates (and in some cases, the exact value) of the quantity $\cat(C_{\Delta_kSP^n(N_g)})$. This enables us to compare the values $\TC_k(SP^n(N_g))$ and $\cat(C_{\Delta_kSP^n(N_g)})$, and to show that they agree in a number of cases.

We will review some standard and useful facts about the symmetric products of non-orientable surfaces, topological complexity, Lusternik--Schnirelmann category, and zero-divisor cup length in Section~\ref{sec: prelims}. For now, we describe the exact formulae for the $k$-th $\Z_2$-zero-divisor cup length of the $n$-th symmetric product of the closed non-orientable surface of genus $g\ge 1$ for each $k\ge 2$ and $n\ge 1$, and their application towards the computation of $\TC_k(SP^n(N_g)$ and $\cat(C_{\Delta_kSP^n(N_g)})$. This cup length will be denoted by $\zcl_k(SP^n(N_g))$ and studied in detail in Section~\ref{sec: zero-divisor}. For simplicity, we divide the description of $\zcl_k(SP^n(N_g))$ into a few distinct cases.

The simplest case is when $k=2$. To describe $\zcl_2(SP^n(N_g))$, it turns out that one only needs to know the consecutive $2$-powers that optimally bound $n$ from above and below. Namely, for any given integer $n\ge 1$, there exists $e\ge 0$ such that $2^e\le n<2^{e+1}$, and then $\zcl_2(SP^n(N_g))$ can be described entirely in terms of $g$, $n$, and $e$, depending only on how $g$ interacts with the gap between $n$ and $2^e$. The precise statement, which will be proven in Section~\ref{sec: k=2}, is as follows. 

\begin{thm}\label{thm: zcl k=2}
  Let $2^e\le n< 2^{e+1}$ for some $e\ge 0$. Then for $g\ge 1$, we have that 
  \[\zcl_2(SP^n(N_g))=\begin{cases}
      2^{e+2}+g-2 & \text{ if } g\le 2n-2^{e+1}+1;
      \\
      2^{e+1}+2n-1 & \text{ if } g\ge 2n-2^{e+1}+1.
  \end{cases}
  \]
\end{thm}

Cases with $k\ge 3$ require more machinery. In particular, unlike the case $k=2$, one needs to better understand the binary expansion of $n$ (actually that of $2n$) to describe $\zcl_k(SP^n(N_g))$. To that end, we follow the notations of~\cite{Dav} and define the following.

\begin{definition}\label{defn: notation}
    Suppose the binary expansion of $m\ge 1$ (where $2^e\le m<2^{e+1}$ for some $e\ge 0$) is given by $m=\delta_e2^e+\delta_{e-1}2^{e-1}+\cdots +\delta_12^1+\delta_02^0$, where $\delta_j\in \{0,1\}$. Define the set
    \[
    S(m): = \left\{ i\in\{0,\ldots,e\} \ \middle| \ \delta_i=\delta_{i-1}=1 \text{ and } \delta_{i+1}=0 \right\}.
    \]
    For $0\le i \le e$, define also the number
    \[
    Z_i(m):=\sum_{j=0}^i\left(1-\delta_j\right)2^j.
    \]
\end{definition}

In other words, $Z_i(m)$ is the sum of all $2^j$ with $j\le i$ such that the $j$-th position in the binary expansion of $m$ corresponds to $0$, and $i\in S(m)$ if a sequence of at least two consecutive $1$'s begins at the $i$-th position in the binary expansion of $m$.

To describe $\zcl_k(SP^n(N_g))$ for $k\ge 3$, we will need to take $m=2n$ above. Our description will utilize that of the $k$-th zero-divisor cup length of even-dimensional real projective spaces from the work of Davis, see~\cite{Dav}. Before recalling that, we define the following.

\begin{definition}\label{gapdefinido}
    For a finite-dimensional CW complex $X$ and $k\ge 2$, define the $k$\emph{-th zero-divisors gap of $X$}, denoted $\gap_k(X)$, as follows:
    \[
    \gap_k(X):=k\cdot\dim(X)-\zcl_k(X).
    \]
\end{definition}

In particular,
\begin{equation}
\label{lasuma}
\gap_k(SP^n(N_g))+\zcl_k(SP^n(N_g))=2nk.
\end{equation}
The quantities $\zcl_k(P^{2n})$ and $\gap_k(P^{2n})$ are completely understood. Indeed, it was shown in~\cite{Dav} that
\begin{equation}\label{eq: davis gap main}
\gap_k(P^{2n})=\max_{i\in S(2n)}\left\{0,2^{i+1}-1-kZ_i(2n)\right\}.    
\end{equation}
The description of $\zcl_k(SP^n(N_g))$, and hence that of $\gap_k(SP^n(N_g))$, for $k\ge 3$ will be divided in various parts. The case of odd $k\ge 3$ is simpler. In the case when $k\ge 4$ is even, we give three formulae for $\zcl_k(SP^n(N_g))$ depending on how $g$ compares with the ratios $\tfrac{\gap_k(P^{2n})}{k}$ and $\tfrac{\gap_k(P^{2n})}{k-2}$. The precise statement is as follows.  

\begin{thm}\label{thm: zcl k ge 3}
    Let $k\ge 3$ and $n,g\ge 1$ be integers.
    \begin{enumerate}[(A)]
        \item If $k\ge 3$ is odd, or if $k\ge 4$ is even and $g\le \left\lfloor \tfrac{\gap_k(P^{2n})}{k}\right\rfloor+1$, then we have that
 \[
         \zcl_k(SP^n(N_g))=2nk-\max_{i\in S(2n)}\left\{0,2^{i+1}-1-kZ_i(2n)-(k-1)(g-1)\right\}.
 \]\label{1st}
    \item If $k\ge 4$ is even and $\left\lfloor \tfrac{\gap_k(P^{2n})}{k}\right\rfloor+2\le g\le\left\lfloor\tfrac{\gap_k(P^{2n})}{k-2}\right\rfloor+1$, then we have that
    \[
    \zcl_k(SP^n(N_g))=2nk+\left(\frac{k}{2}-1\right)\left(g-1\right)-\frac{\gap_k(P^{2n})+1}{2}.
    \]\label{2nd}
    \item If $k\ge 4$ is even and $g\ge\left\lfloor\tfrac{\gap_k(P^{2n})}{k-2}\right\rfloor+2$, then we have that 
    \[
    \zcl_k(SP^n(N_g))=2nk.
    \]\label{3rd}
\end{enumerate}
\end{thm}

In view of~\eqref{lasuma}, Theorem~\ref{thm: zcl k ge 3} is equivalent to the following theorem that will be proven in Subsection~\ref{subsec: odd k proof} and Section~\ref{sec: k even}.

\begin{thm}\label{thm: gap k ge 3}
    Let $k\ge 3$ and $n,g\ge 1$ be integers.
    \begin{enumerate}[(a)]
        \item If $k\ge 3$ is odd, or if $k\ge 4$ is even and $g\le \left\lfloor \tfrac{\gap_k(P^{2n})}{k}\right\rfloor+1$, then we have that
 \[
    \gap_k(SP^n(N_g))=\max_{i\in S(2n)}\left\{0,2^{i+1}-1-kZ_i(2n)-(k-1)(g-1)\right\}.
    \]    \label{first}
    \item If $k\ge 4$ is even and $\left\lfloor \tfrac{\gap_k(P^{2n})}{k}\right\rfloor+2\le g\le\left\lfloor\tfrac{\gap_k(P^{2n})}{k-2}\right\rfloor+1$, then we have that
    \[
    \gap_k(SP^n(N_g))=\frac{\gap_k(P^{2n})+1}{2}-\left(\frac{k}{2}-1\right)\left(g-1\right).
    \]\label{second}
    \item If $k\ge 4$ is even and $g\ge\left\lfloor\tfrac{\gap_k(P^{2n})}{k-2}\right\rfloor+2$, then we have that 
    \[
    \gap_k(SP^n(N_g))=0.
    \]\label{third}
\end{enumerate}
\end{thm}

Note that item~\eqref{second} is empty precisely when $\bigl\lfloor\tfrac{\gap_k(P^{2n})}{k}\bigr\rfloor=\bigl\lfloor\tfrac{\gap_k(P^{2n})}{k-2}\bigr\rfloor$, in which case items~\eqref{first} and~\eqref{third} are complementary.

Since $\zcl_k(SP^n(N_g))$ is a lower bound to both the $k$-th sequential topological complexity of $SP^n(N_g)$ and the Lusternik--Schnirelmann category of the homotopy cofiber of the $k$-th diagonal map $SP^n(N_g)\to (SP^n(N_g))^k$, our results on the former invariant have the following implications for the latter two invariants (see Section~\ref{sec: observations2} for a proof). 

\begin{corollary}\label{cor: tc_k}
   Let $k\ge 2$ and $n,g\ge 1$ be integers.
   \begin{enumerate}[(i)]
       \item If $k\ \ge \ 3$ is odd and $g\ \ge \ \left\lceil\tfrac{\gap_k(P^{2n})}{k-1}\right\rceil+1$, or if $k\ \ge\ 4$ is even and $g\ge\left\lfloor\tfrac{\gap_k(P^{2n})}{k-2}\right\rfloor+2$, or if $n=2^e$ for some $e\ge 0$ and $k\ge 3$, then we have that 
    \[
\TC_k(SP^n(N_g))=2nk=\cat(C_{\Delta_kSP^n(N_g)}).
    \] \label{2second}
    \item If $2^e\le n<2^{e+1}$ for some $e\ge 1$, then we have that
    \[
2^{e+2}-1\le\cat(C_{\Delta_2SP^n(N_g)})\le\TC_2(SP^n(N_g))\le 4n-1.
    \]
    In particular, the last three inequalities are equalities if $n=2^e$.
    \label{1first}
   \end{enumerate}
\end{corollary}

We note that in the case $g=1$ (i.e., when $SP^n(N_1)\cong P^{2n})$, our results recover those on the zero-divisor cup lengths of $P^{2n}$, see~\cite{Dav} (and also~\cite{CAG+} for $k=2$). Similarly, in the case $n=1$ (i.e., when $SP^1(N_g)=N_g$), our results recover the computation for $k\ge 3$ in~\cite[Section~5]{GGGL}, as well as the folklore fact that $\zcl_2(N_g)=3$ (see, for instance,~\cite[proof of Theorem~3.4.4]{Dr2}).

It is known from~\cite[Proposition~2.1~(2)]{CAC} that, for a space $X$ with abelian fundamental group, $\TC_k(X)$ attains its maximal possible value $k\cdot \dim(X)$ if and only if $\cat(C_{\Delta_kX})$ does so. Since $\pi_1(SP^n(N_g))$ is abelian for $n\ge 2$, Corollary~\ref{cor: tc_k} provides an explicit, new infinite family of examples for which the equalities $\TC_k(X)=\cat(C_{\Delta_kX})=k\cdot \dim(X)$ hold for each $k\ge 2$ (in fact, the previously-conjectured equality $\TC_2(X)=\cat(C_{\Delta_2X})=\TC^M(X)$ holds in these cases for the \emph{monoidal topological complexity} of~\cite{IS}, see Remark~\ref{rem: monoidal}). That the equality $\TC_2(X)=\cat(C_{\Delta_2X})$ holds for a few classes of spaces (such as spheres, path-connected $H$-spaces, simply connected closed symplectic manifolds, and real projective spaces) was shown in~\cite{GV}, while it follows from~\cite{Dr2,Dr3,CV1} that $\cat(C_{\Delta_2N_g})=3<4=\TC_2(N_g)$ for $g\ge 2$. In light of the latter inequality, the condition $e\ge 1$ in Corollary~\ref{cor: tc_k}~\eqref{1first} cannot be waived for $g\ge 2$.

\subsection*{Notations and conventions}
In this paper, we consider (co)homology with $\Z_2$-coefficients, unless specified otherwise. Similarly, zero-divisor cup lengths are considered only with $\Z_2$-coefficients. The tensor product is always taken over the ring of integers and denoted $\otimes$. The notation $\smile$ is omitted in writing cup products, i.e., we write $\alpha\beta$ for a cup product $\alpha\smile \beta$. We use the term \emph{map} to refer to both continuous functions and algebra homomorphisms. Unless otherwise noted, all spaces are assumed to be connected.

\section{Preliminaries and related work}\label{sec: prelims}

\subsection{Symmetric products of closed non-orientable surfaces} Let $N_g$ denote the closed non-orientable surface of genus $g\ge 1$, and $\Sigma_n$ the symmetric group on $n\ge 1$ symbols. A permutation $\sigma\in \Sigma_n$ permutes the factors in the Cartesian product $(N_g)^n$. The orbit space of this action is the $n$-th symmetric product of $N_g$, denoted $SP^n(N_g)$. It is well-known that $SP^n(N_g)$ is a closed smooth non-orientable $2n$-manifold. For example, $SP^1(N_g)=N_g$ and $SP^n(N_1)\cong P^{2n}$, see~\cite{KS}.

For $n>1\leq g$, it is a standard fact that $\pi_1(SP^n(N_g))\cong H_1(N_g;\Z)\cong \Z^{g-1}\oplus\Z_2$, see, for instance,~\cite{KT}. In particular, $SP^n(N_g)$ is not aspherical. However, as shown in~\cite{Ja}, $SP^n(N_g)$ is always essential in the sense of Gromov (since its Lusternik--Schnirlemann category is $2n$),  and its universal cover is $(n-1)$-connected. Because of fundamental group reasons, $SP^n(N_g)$ does not support a Riemannian metric of non-positive sectional curvature due to the Cartan--Hadamard theorem, and for $g\ge 2n+1$, it also does not support a Riemannian metric of non-negative Ricci curvature due to Bochner's theorem on the first Betti number. The cohomology ring of $SP^n(N_g)$ is known due to~\cite{KS}; we recall the description in Subsection~\ref{subsec: coho ring}. For a survey on these manifolds, we refer the reader to~\cite{BGZ}.

\subsection{Zero-divisor cup length and sequential topological complexity}\label{subsec: zero div} 
For a CW complex $X$ and integer $k\ge 2$, let $\Delta_{k,X}\colon X\to X^k$ denote the diagonal map. In cohomology with coefficients in a ring $R$, this induces an algebra homomorphism $\Delta_{k,X}^*\colon H^*(X^k;R)\to H^*(X;R)$. A cohomology class $\alpha\in \Ker(\Delta_{k,X}^*)$ is called a \emph{$k$-th $R$-zero-divisor} of $X$, and the length of the longest non-zero cup product of such zero-divisors is called the \emph{$k$-th $R$-zero-divisor cup length} of $X$, denoted $\zcl_{k,R}(X)$. We reserve the notation $\zcl_k(X):=\zcl_{k,\Z_2}(X)$. For this and the following terminology and results, we refer the reader to~\cite{Far,Far2,Ru,BGRT}.

The \emph{$k$-th sequential topological complexity} of $X$, denoted $\TC_k(X)$, is defined to be the smallest non-negative integer $n$ such that $X^k=\bigcup_{i=1}^{n+1}U_i$, where each $U_i\subset X^k$ is open and admits a map $s_i\colon U_i\to X^{[0,1]}$ satisfying
\[
s_i(x_1,\ldots,x_k)\left(\frac{j-1}{k-1}\right)=x_j
\]
for each $j\in \{1,\ldots,k\}$ and $(x_1,\ldots,x_k)\in U_i$. For future reference, we shall call $s_i$ a \emph{motion planner} on $U_i$.

A homotopy invariant closely related to, and much older than, $\TC_k(X)$ is the \emph{Lusternik--Schnirelmann category} (or LS-category for short) of $X$, which is denoted $\cat(X)$ and defined to be the smallest non-negative integer $m$ such that $X=\bigcup_{j=1}^{m+1}V_j$, where each $V_j$ is open and null-homotopic in $X$.

The following well-known result connects these invariants.
\begin{thm}\label{thm: tc prelim}
For a CW complex $X$, ring $R$, and integer $k\ge 2$, we have that $\max\{\zcl_{k,R}(X),\cat(X^{k-1})\}\le\TC_k(X)\le\cat(X^k)$. 
\end{thm}

It was shown in~\cite{Ja} that $\cat((SP^n(N_g)^m)=2mn$ for $m,n,g\ge 1$. Therefore, it follows that $2n(k-1)\le\TC_k(SP^n(N_g))\le 2nk$. Because $SP^n(N_1)\cong P^{2n}$, the integer $\TC_k(SP^n(N_1))$ coincides with one less than the minimal number of motion planning rules required to rotate a line in $\R^{2n+1}$ that is fixed along a revolving joint at the origin via $k$ positions in $\R^{2n+1}$. For $k=2$, this integer is the immersion dimension of $P^{2n}$ (see~\cite{FTY}); for $k\ge 3$, this integer is $2nk$ in some cases (depending on the combinatorics of $k$ and $n$) and undetermined in others, see~\cite{Dav}.

\subsection{LS-category of the homotopy cofiber of the diagonal map}
For a CW complex $X$ and integer $k\ge 2$, let $C_{\Delta_kX}$ denote the cofiber of the diagonal map $\Delta_{k,X}\colon X\to X^k$. Indeed, $C_{\Delta_kX}=X^k\cup_{\Delta_{k,X}}CX$ is the mapping cone of $\Delta_{k,X}$. Since the image $\Delta_kX$ of $X$ under the cofibration $\Delta_{k,X}$ is closed in $X^k$, we have that $C_{\Delta_kX}\simeq X^k/\Delta_kX$, see~\cite{Ha}. Consequently, because $(X^k,\Delta_{k}X)$ is a \emph{good CW-pair} (in the sense of~\cite{Ha}), we see that
\[
H^s(C_{\Delta_kX};R) \cong H^s(X^k,\Delta_kX;R)    
\]
for any $s\ge 1$ and ring $R$.

The following fact regarding the LS-category of $C_{\Delta_kX}$ is a direct generalization of the well-known case with $k=2$. The proof is short, and we record it for completeness.

\begin{thm}\label{thm: cat prelim}
    For a finite-dimensional CW complex $X$, ring $R$, and integer $k\ge 2$, we have that $\zcl_{k,R}(X)\le \cat(C_{\Delta_kX})\le\min\{\TC_k(X)+1, k\cdot\dim(X)\}$.
\end{thm}

\begin{proof}
    First, we prove $\cat(C_{\Delta_kX})\le \TC_k(X)+1$. If $\TC_k(X)=m$, then we can write $X^k=\bigcup_{i=1}^{m+1}U_i$, where each $U_i\subset X^k$ is open and admits a motion planner $s_i\colon U_i\to X^{[0,1]}$. It then turns out that for each $i$, the inclusion $U_i\hookrightarrow X^k$ is homotopic to a map with values in $\Delta_kX$. Indeed, for any fixed $i$, using maps $\phi_r\colon [0,1]\to [0,1]$ defined by
    \[
    \phi_r(t):=\frac{t(k-r)+r-1}{k-1}
    \]
    for $1\le r \le k$, we can define the required deformation $H_i\colon U_i\times I\to X^k$ of $U_i$ as
    \[ H_i(x_1,\ldots,x_k,t):=\left(s_i(x_1,\ldots,x_k)(\phi_1(t)),\ldots,s_i(x_1,\ldots,x_k)(\phi_k(t))\right).
    \]
    A straightforward generalization of the proof of~\cite[Lemma~18.3]{Far2} then yields $\cat(C_{\Delta_kX})\le\TC_k(X)+1$. The other upper bound $\cat(C_{\Delta_kX})\le k\cdot\dim(X)$ holds because $C_{\Delta_kX}\simeq X^k/\Delta_kX$ and $\dim(X^k/\Delta_kX)\le k\cdot\dim(X)$. For the lower bound, suppose $\zcl_{k,R}(X)=n$, so that for $1\le i \le n$, there exist classes $\alpha_i\in H^{s_i}(X^k;R)$ satisfying $\alpha_1\cdots\alpha_n\ne 0$ and $\Delta_{k,X}^*(\alpha_i)=0$. For each $i$, the long exact sequence for the pair $(X^k,\Delta_kX)$ gives
    \[
    \cdots\to  H^{s_i}(X^k,\Delta_kX;R)\xlongrightarrow{j} H^{s_i}(X^k;R)\xlongrightarrow{\Delta_{k,X}^*} H^{s_i}(X;R)\to\cdots,
    \]
    from which we deduce that there exists $\beta_i\in H^{s_i}(X^k,\Delta_kX;R)\cong H^{s_i}(C_{\Delta_kX};R)$ such that $j(\beta_i)=\alpha_i$. It follows from the naturality of the cup product that $\beta_1\cdots\beta_n\ne 0$ in $H^*(C_{\Delta_kX};R)$. Since the $R$-cup length of any CW complex is a lower bound to its LS-category (see, for instance,~\cite[Proposition~1.5]{CLOT}), we have that $\zcl_{k,R}(X)=n\le\cat(C_{\Delta_kX})$.
\end{proof}

\subsection{Comparison from computations for $SP^n(M_g)$ and $P^{2n}$} 
Similar to the symmetric products of closed non-orientable surfaces, one has the $n$-th symmetric product of the closed orientable surface $M_g$ of genus $g\ge 0$, denoted $SP^n(M_g)$. This is a closed smooth orientable $2n$-manifold which has been well-studied in topology and geometry---see, for instance,~\cite{Mac,BGZ,KS,DCDJ}. There are several fundamental differences between $SP^n(M_g)$ and $SP^n(N_g)$; we highlight some of these with a $\TC$-flavor coming from Theorems~\ref{thm: zcl k=2} and~\ref{thm: zcl k ge 3} and the results of~\cite{Ja,DCDJ}.

\begin{enumerate}[(1)]
    \item The quantity $\zcl_{k,\Q}(SP^n(M_g))$ is much simpler to compute than $\zcl_k(SP^n(N_g))$. The assessment of the former quantity is based on Macdonald's description (see~\cite{Mac}) of the ring $H^*(SP^n(M_g);\Q)$, and the answer depends only on how $g$ compares with $n$. On the other hand, the computation of $\zcl_k(SP^n(N_g))$ uses Kallel--Salvatore's description (see Section~\ref{sec: cohomo}) of the ring $H^*(SP^n(N_g);\Z_2)$ and Davis's analysis of the case $g=1$ in~\cite{Dav}; the answer depends heavily on the binary expansion of $n$, on both the magnitude and parity of $k$, as well as on the comparison of $g$ with two ratios---indicated in the paragraph containing~\eqref{eq: davis gap main}---involving the invariant $\gap_k(P^{2n})$.

\item The value $\TC_{k}(SP^n(M_g))$ is known for each $k$, $n$ and $g$, while $\TC_k(SP^n(N_g))$ is unknown in many cases. Indeed, the former value is always equal to the quantity $\zcl_{k,\Q}(SP^n(M_g))$, while we (the authors) are aware of the equality $\zcl_k(SP^n(N_g))=\TC_k(SP^n(N_g))$ only when this number is maximal possible (see Corollary~\ref{cor: tc_k}). These comments also apply to the invariant $\cat(C_{\Delta_kX})$ for $X\in\{SP^n(M_g),SP^n(N_g)\}$. As a result, the computation of the $\TC$-generating function (see~\cite{FO,FKS}) of $SP^n(M_g)$ is simpler than that of $SP^n(N_g)$. In the former case, the $\TC$-generating rational function is explicitly described, with the integral polynomial in its numerator being a quadratic. In contrast, the numerator of the $\TC$-generating rational function of $SP^n(N_g)$ is mostly a cubic polynomial, and in general, we only have bounds on its degree, see Theorem~\ref{thm: rationality} below.
\end{enumerate}

As previously noted, the symmetric products of closed non-orientable surfaces generalize even-dimensional real projective spaces. The description of $\zcl_k(P^{2n})$ from~\cite{Dav} is simpler than that of $\zcl_k(SP^n(N_g))$. Indeed, our results show that for non-trivial genera (i.e., $g\ne 1$), the value $\zcl_k(SP^n(N_g))$ depends on both the magnitude of $g$ and the parity of $k$ --- see Example~\ref{ex: example} for an explicit example, and also the comment following the statement of Lemma~\ref{prop: final}.

All in all, the study of some homotopy invariants of $SP^n(N_g)$, such as the zero-divisor cup lengths and sequential topological complexities, is more involved than those of $SP^n(M_g)$ and $P^{2n}$.

\section{Consequences and observations}\label{sec: observations1}
In this section, we study the consequences of our main results for the zero-divisor cup lengths. We begin by looking at a simple example.

\begin{ex}\label{ex: example}
Suppose $n=51$, so that $2n=102=2^6+2^5+2^2+2^1$. We have from~\cite{Dav} that $\gap_k(P^{102})$ is equal to $127-25k$ if $2\le k \le 5$, $7-k$ if $5\le k \le 7$, and $0$ otherwise. Since $2,6\in S(102)$, $Z_2(102)=1$, and $Z_6(102)=25$, we get from Theorems~\ref{thm: zcl k=2} and~\ref{thm: gap k ge 3} the following computations for $SP^{51}(N_g)$:
\[
\gap_k(SP^{51}(N_g))=\begin{cases}
78-g & \text{ if }k=2,g\le 39;
\\
39 & \text{ if }k=2,g\ge 39;
\\
54-2g & \text{ if }k=3,g\le 26;
\\
30-3g & \text{ if }k=4,g\le 7;
\\
15-g & \text{ if }k=4,8\le g\le 14;
\\
2 & \text{ if }k=5,g=1;
\\
1 & \text{ if }k=6,g=1;
\\
0 & \text{ otherwise}.
\end{cases}
\]
It is clear from this example that the computation of $\gap_k(SP^n(N_g))$ is strictly more complicated than that of $\gap_k(P^{2n})$ in most cases.
\end{ex}

The following statement is a straightforward consequence of Theorem~\ref{thm: zcl k=2}. 

\begin{corollary}
    Let $2^e\le n<2^{e+1}$ for some $e\ge 0$. Then
    \[
   \lim_{g\to\infty} \zcl_2(SP^n(N_g))=2^{e+1}+2n-1.
    \]
\end{corollary}

By Corollary~\ref{cor: tc_k}, this limit attains its maximal possible value $\TC_2(SP^n(N_g))$ when $n=2^e\ge 2$, but is less than $4n-1$ otherwise. In stark contrast to this case, we have the following when $k\ge 3$.

\begin{corollary}
    Fix $n\ge 1$. Then for $k\ge 3$, we have that
    \[
    \lim_{g\to \infty}\zcl_k(SP^n(N_g))=2nk \ \text{ and } \ \lim_{g\to \infty}\TC_k(SP^n(N_g))=2nk.
    \]
\end{corollary}

\begin{proof}
    The assertion for the limit on the left follows for each $k\ge 3$ directly by parts ~\eqref{1st} and~\eqref{3rd} of Theorem~\ref{thm: zcl k ge 3}. To justify the limit on the right, it suffices to note that $\zcl_k(SP^n(N_g))\le\TC_k(SP^n(N_g))\le2nk$ for each $k$ and $g$.
\end{proof}

We show in Corollary~\ref{cor: zcl with g changing} that $\zcl_k(SP^n(N_g))\le \zcl_k(SP^n(N_{g+1}))$ for fixed $n$ and $k$. Theorems~\ref{thm: zcl k=2} and~\ref{thm: zcl k ge 3} help us obtain an upper bound on $\zcl_k(SP^n(N_{g+1}))$ in terms of $\zcl_k(SP^n(N_g))$ in all cases. This is achieved by precisely controlling the linear growth of the $k$-th zero-divisor cup length (or equivalently, the linear decay of the $k$-th zero-divisors gap) when the genus $g$ increases by $1$.

\begin{proposition}\label{prop: zcl difference}
    Let $g\ge 1$ and $2^e\le n<2^{e+1}$ for some $e\ge 0$.
    \begin{enumerate}
        \item If $k=2$, then we have that $\zcl_2(SP^n(N_{g+1}))- \zcl_2(SP^n(N_g))\le 1$, where equality holds precisely when $g\le 2n-2^{e+1}$.\label{diff one}
        \item If $k\ge 3$ is odd, then $\zcl_k(SP^n(N_{g+1}))- \zcl_k(SP^n(N_g))\le k-1$, which is an equality if $\zcl_k(SP^n(N_{g+1}))\ne 2nk$, and is equal to $0$ if $g\ge\left\lceil\tfrac{\gap_k(P^{2n})}{k-1}\right\rceil+1$.\label{diff two}
        \item If $\ k\ \ge\ 4 \ $ is even and $\ g\ \le\ \left\lfloor\tfrac{\gap_k(P^{2n})}{k}\right\rfloor$, then we have the equality $\zcl_k(SP^n(N_{g+1}))- \zcl_k(SP^n(N_g))=k-1$.\label{diff three}
        \item If $k \ge 4$ is even and $g= \left\lfloor\tfrac{\gap_k(P^{2n})}{k}\right\rfloor+1$, then we have the equality
        \[
       \zcl_k(SP^n(N_{g+1}))- \zcl_k(SP^n(N_g))=\begin{cases}
           \frac{k+r-3}{2} & \ \text{ if } k\le 2g+r;
           \\
           g+r-1 & \ \text{ if } k\ge 2g+r+1,
       \end{cases} 
        \]
    where $r$ is the residue obtained after dividing $\gap_k(P^{2n})$ by $k$.
       \item If $k\ge 4$ is even and $\left\lfloor\tfrac{\gap_k(P^{2n})}{k}\right\rfloor+2\le g \le \left\lfloor\tfrac{\gap_k(P^{2n})}{k-2}\right\rfloor$, then we have the equality $\zcl_k(SP^n(N_{g+1}))-\zcl_k(SP^n(N_g))=\tfrac{k}{2}-1$.\label{diff five}
        \item If $k \ge 4$ is even and $g \ \ge \ \left\lfloor\tfrac{\gap_k(P^{2n})}{k-2}\right\rfloor+1$, then we have the equality $\ \zcl_k(SP^n(N_{g+1}))\ - \ \zcl_k(SP^n(N_g))=\gap_k(SP^n(N_{g}))$, which is positive precisely when $g =\left\lfloor\tfrac{\gap_k(P^{2n})}{k-2}\right\rfloor+1$ and $\gap_k(P^{2n})>0$.\label{diff six}
        \end{enumerate}
\end{proposition}

\begin{proof}
As $\zcl_k(SP^n(N_{g+1}))-\zcl_k(SP^n(N_g))=\gap_k(SP^n(N_g))-\gap_k(SP^n(N_{g+1}))$, we will prove this proposition for the latter quantity for our convenience. Clearly, the assertions in~\eqref{diff one} and~\eqref{diff five} follow immediately from Theorems~\ref{thm: zcl k=2} and~\ref{thm: gap k ge 3}~\eqref{second}, respectively. In the other cases, we proceed as follows.

\begin{enumerate}[(2)]
    \item We are done if $\gap_k(SP^n(N_{g}))=0$, because then $\gap_k(SP^n(N_{g+1}))=0$ due to Corollary~\ref{cor: zcl with g changing}. Note that $\gap_k(SP^n(N_g))=\max\{0,\gap_k(P^{2n})-(k-1)(g-1)\}$ by ~\eqref{eq: davis gap main} and Theorem~\ref{thm: gap k ge 3}~\eqref{first}. So,
    \[
    g\ge\left\lceil\frac{\gap_k(P^{2n})}{k-1}\right\rceil+1\implies (g-1)(k-1)\ge\gap_k(P^{2n})\implies \gap_k(SP^n(N_g))=0.
    \]
Thus, as argued above, the difference is $0$ in this case. Next, let us assume $0<\gap_k(SP^n(N_{g}))\le\gap_k(P^{2n})$. Then we have from Theorem~\ref{thm: gap k ge 3}~\eqref{first} that
\begin{align*}
    0<\gap_k(SP^n(N_g))&=\max\left\{0,\gap_k(P^{2n})-(k-1)(g-1) \right\}
    \\
    & =\gap_k(P^{2n})-(k-1)(g-1).
\end{align*}
If $\zcl_k(SP^n(N_{g+1}))\ne 2nk$, then $0<\gap_k(SP^n(N_{g+1}))$ and so, in analogy with the above, we get $\gap_k(SP^n(N_{g+1}))=\gap_k(P^{2n})-(k-1)g$ . Therefore, the equality with $k-1$ follows in this case. Otherwise, if $\gap_k(SP^n(N_{g+1}))=0$, then $\gap_k(P^{2n})\le(k-1)g$, which gives $\gap_k(SP^n(N_g))\le k-1$. So, $k-1$ is indeed the general upper bound. 
\end{enumerate}

\begin{enumerate}[(3)]
\item The hypothesis on $g$ gives $\gap_k(P^{2n})>(k-1)g$. So, in analogy with the proof of~\eqref{diff two}, we get that $0<\gap_k(P^{2n})-(k-1)g=\gap_k(SP^n(N_{g+1}))$, and then the asserted equality follows similarly.
\end{enumerate}

\begin{enumerate}[(4)]
   \item We can write $\gap_k(P^{2n})=(g-1)k+r$, where $0\le r<k$. By Theorem~\ref{thm: gap k ge 3}~\eqref{first}, $\gap_k(SP^n(N_g))$ is equal to
   \[
   \max\left\{0,\gap_k(P^{2n})-(k-1)(g-1)\right\}=\max\left\{0,g+r-1\right\}=g+r-1
   \]
   since $g+r-1\ge 0$. If is not difficult to check that $g\le\left\lfloor\tfrac{\gap_k(P^{2n})}{k-2}\right\rfloor$ if and only if $k\le 2g+r$. So, if $k\le 2g+r$, then Theorem~\ref{thm: gap k ge 3}~\eqref{second} gives
   \[
   \gap_k(SP^n(N_{g+1}))=\frac{(g-1)k+r+1}{2}-\left(\frac{k}{2}-1\right)g=\frac{2g+r+1-k}{2},
   \]
   which means the difference is
   \[
    g+r-1-\frac{2g+r+1-k}{2}=\frac{k+r-3}{2}\le k-2.
   \]
   Otherwise, if $k\ge 2g+r-1$, then $\gap_k(SP^n(N_{g+1}))=0$ by Theorem~\ref{thm: gap k ge 3}~\eqref{third}, and so the difference is just $g+r-1$, which does not exceed $g+k-2$.
\end{enumerate}

\begin{enumerate}[(6)]
    \item Here, we have $\gap_k(SP^n(N_{g+1}))=0$ because of Theorem~\ref{thm: gap k ge 3}~\eqref{third}. Therefore, the asserted equality holds by definition. Note that $\gap_k(SP^n(N_{g}))=0$ when $g\ge \left\lfloor \tfrac{\gap_k(P^{2n})}{k-2} \right\rfloor+2$, and so the difference is $0$ in this range. However, if $g=\left\lfloor \tfrac{\gap_k(P^{2n})}{k-2} \right\rfloor+1$, which implies $g\le\tfrac{\gap_k(P^{2n})}{k-2}+1$, then we claim that 
\[
\gap_k(SP^n(N_{g}))=0\iff\gap_k(P^{2n})=0.
\]
The reverse implication is obvious. To prove the forward implication, assume that $\gap_k(SP^n(N_{g}))=0$. If $\left\lfloor\frac{\gap_k(P^{2n})}{k}\right\rfloor+2\le g$, then we get the contradiction
\[
1+\frac{\gap_k(P^{2n})+1}{k-2}=g>\frac{\gap_k(P^{2n})}{k-2}+1,
\]
where the first equality comes from Theorem~\ref{thm: gap k ge 3}~\eqref{second}. Hence, we must have $g\le\left\lfloor\tfrac{\gap_k(P^{2n})}{k}\right\rfloor+1$. In this range, Theorem~\ref{thm: gap k ge 3}~\eqref{first} gives
\[
0=\gap_k(SP^n(N_g))=\max_{i\in S(2n)}\left\{0,2^{i+1}-1-kZ_i(2n)-(k-1)(g-1) \right\}.
\]
If $\gap_k(P^{2n})>0$, then $g>1$ (in view of the equality above) and so the upper bound on $g$ implies $\gap_k(P^{2n})>(k-1)(g-1)$. As argued in the proof of~\eqref{diff two}, we then get $0<\gap_k(P^{2n})-(k-1)(g-1)=\gap_k(SP^n(N_g))$, which is a contradiction. So, $\gap_k(P^{2n})=0$ must be the case. Thus, in this case, the difference, namely $\gap_k(SP^n(N_g))$, is positive if and only if $\gap_k(P^{2n})>0$. 
\end{enumerate}
\end{proof}

\section{Topological complexity and the rationality conjecture}\label{sec: observations2}

We begin by proving Corollary~\ref{cor: tc_k}, whose implications we will subsequently discuss in this section.

\begin{proof}[Proof of Corollary~\ref{cor: tc_k}]
We first focus on the case of~\eqref{2second}. The assertion for $k\ge 4$ even is obvious from Theorem~\ref{thm: zcl k ge 3}~\eqref{3rd}. If $k\ge 3$ is odd and $g\ge \left\lceil\tfrac{\gap_k(P^{2n})}{k-1}\right\rceil+1$, then $\gap_k(SP^n(N_g))=0$ holds because Theorem~\ref{thm: gap k ge 3}~\eqref{first} implies the equality 
\[
\gap_k(SP^n(N_g))=\max\left\{0,\gap_k(P^{2n})-(k-1)(g-1)\right\}.
\] 
In view of this, $\zcl_k(SP^n(N_g))=\cat(C_{\Delta_kSP^n(N_g)})=\TC_k(SP^n(N_g))=2nk$ follows from Theorems~\ref{thm: tc prelim} and~\ref{thm: cat prelim}. If $n=2^e$ and $k\ge 3$ is arbitrary, then $\gap_k(P^{2n})=0$ due to~\cite{Dav}. So, we get from parts~\eqref{1st} and~\eqref{3rd} of Theorem~\ref{thm: zcl k ge 3} that
\begin{align*}
    2nk\ge \TC_k(SP^n(N_g)) & \ge \zcl_k(SP^n(N_g))=2nk \ \text{ and} 
    \\
    2nk\ge \cat(C_{\Delta_kSP^n(N_g)}) & \ge \zcl_k(SP^n(N_g))=2nk \ \text{ for all } \ k\ge 3.
\end{align*}
Next, we look at the case of~\eqref{1first}. Recall that $\pi_1(SP^n(N_g))\cong H_1(N_g;\Z)$ is abelian since $n\ge 2$, and that $SP^n(N_g)$ is a path-connected closed non-orientable $2n$-manifold. So,~\cite[Theorem~1.2]{CV2} implies that $\TC_2(SP^n(N_g))$ is not maximal, i.e., it does not exceed $4n-1$. From Theorems~\ref{thm: zcl k=2} and~\ref{thm: tc prelim}, we have that 
\begin{equation}\label{dimvsTC}
\dim(SP^n(N_g))=2n<2^{e+2}-1\le\zcl_2(SP^n(N_g))\le\TC_2(SP^n(N_g)).
\end{equation}
Thus, $\zcl_2(SP^n(N_g))\le \cat(C_{\Delta_2SP^n(N_g)})\le \TC_2(SP^n(N_g))$ holds by Theorem~\ref{thm: cat prelim} and~\cite[Theorem~10~(2)]{GV}, respectively. Finally, if $n=2^e$, then we have $\TC_2(SP^n(N_g))\le 4n-1\le\zcl_2(SP^n(N_g))$, so all inequalities are equalities.
\end{proof}

\begin{remark}
    The result from~\cite{CV2} used above says that $\TC_2(M)<2\dim(M)$ if $M$ is a connected closed non-orientable manifold of dimension greater than 1 and with abelian $\pi_1(M)$. This result does not extend to $\TC_k(M)$ for any $k\ge 3$ in view of Corollary~\ref{cor: tc_k}. 
\end{remark}

\begin{remark}\label{rem: monoidal}
    As observed in~\eqref{dimvsTC}, the inequality $\dim(SP^n(N_g))<\TC_2(SP^n(N_g))$ holds for $n,g\ge 1$, so~\cite[Theorem~2.5]{Dr1} gives $\TC_2(SP^n(N_g))=\TC^M(SP^n(N_g))$. Here, $\TC^M(X)$ stands for the \emph{monoidal topological complexity} of $X$ introduced in~\cite{IS}, which is a quantity equal to either $\TC_2(X)$ or $\TC_2(X)+1$ for a CW complex $X$, and conjectured to be always equal to $\TC_2(X)$, see~\cite{IS}. To the best of our knowledge, another conjecture on $\TC^M$ (see~\cite[Section~4]{Dr1}), that $\TC^M(X)=\cat(C_{\Delta_2X})$, is also still open. Our results give the equalities
    \[
\cat(C_{\Delta_2SP^{2^e}(N_g)})=\TC_2(SP^{2^e}(N_g))=\TC^M(SP^{2^e}(N_g))=2^{e+2}-1
    \]
    for $e,g\ge 1$. Therefore, the conjectures of both Iwase--Sakai and Dranishnikov are verified for infinitely many (non-trivial) symmetric products of closed non-orientable surfaces. The same assertion holds true for all symmetric products of closed orientable surfaces of genus $\ge 1$ due to~\cite[Theorem~6.8]{DCDJ} (see also~\cite{Ja}).
\end{remark}

We now verify the \emph{rationality conjecture} of Farber and Oprea~\cite{FO} for the $\TC$-generating function of $SP^n(N_g)$ for $n,g\ge 1$.

\begin{thm}\label{thm: rationality}
    Let $n,g\ge 1$ be integers. Then the $\TC$-generating function of $SP^n(N_g)$, namely the formal power series 
    \[
    f_{SP^n(N_g)}(t)=\sum_{k=1}^{\infty}\ \TC_{k+1}(SP^n(N_g))\ t^k,
    \]
    represents a rational function of the form 
    \[
    \frac{P_{SP^n(N_g)}(t)}{(1-t)^2},
    \]
    where $P_{SP^n(N_g)}(t)$ is an integral polynomial whose value at $t=1$ is $\cat(SP^n(N_g))$. Moreover, the degree of $P_{SP^n(N_g)}(t)$ does not exceed
    \[
        \left\lfloor \frac{\gap_2(P^{2n})}{g-1} \right\rfloor + 3
    \]
    when $g\ge 2$. In fact, the degree of $P_{SP^n(N_g)}(t)$ does not exceed $3$ if either of the following is true:
    \begin{enumerate}
     \item $g\ge \left\lfloor \tfrac{\gap_2(P^{2n})}{2} \right\rfloor+2$ and $n\ge 1$ is arbitrary;
        \item $n=2^e$ for some $e\ge 0$ and $g\ge 1$ is arbitrary.
    \end{enumerate}
\end{thm}

Before supplying a proof, we make some remarks. Recall from~\cite{CAG+} that $\gap_2(P^{2n})=4n-2^{e+2}+1>0$ if $2^e\le n <2^{e+1}$ for some $e\ge 0$. Therefore, the condition in part (1) is $g\ge 2n-2^{e+1}+2$. Note that the rationality conjecture for $g=1$ can be directly verified using~\cite[Proposition 1.4]{Dav}. Indeed, one has $\TC_k(P^{2n})=\zcl_k(P^{2n})= 2nk$ if $k\ge 2^{\ell+1}-1$, where $\ell$ is the length of the longest string of consecutive $1$’s in the binary expansion of $n$. Since $\mathsf{cl}(P^{2n})=\cat(P^{2n})$ is well-known (see, for instance,~\cite{CLOT}), one can apply~\cite[Theorem~1]{FKS} to complete the verification. Such arguments work for $g\ge 2$ as well (see Remark~\ref{rem: on different zcls}); however, we will explicitly verify the rationality conjecture for $g\ge 2$ by giving a relatively better upper bound on the degree of $P_{SP^n(N_g)}(t)$.

\begin{proof}[Proof of Theorem~\ref{thm: rationality}]
    Fix $n\ge 1$ and $g\ge 2$. Note that 
    \[
    k\ge\left\lfloor\frac{\gap_2(P^{2n})}{g-1}\right\rfloor+3\iff k >\frac{\gap_2(P^{2n})}{g-1}+2\iff g \ge\left\lfloor \frac{\gap_2(P^{2n})}{k-2}\right\rfloor+2.
    \]
    The formula in~\cite[Theorem 1.6]{Dav} implies $\gap_2(P^{2n})\ge\gap_k(P^{2n})$ for $k\ge 3$. In view of this and the above implications, we see that
    \[
    k\ge\left\lfloor\frac{\gap_2(P^{2n})}{g-1}\right\rfloor+3\implies g\ge\left\lfloor \frac{\gap_k(P^{2n})}{k-2}\right\rfloor+2\ge     \left\lceil \frac{\gap_k(P^{2n})}{k-1}\right\rceil+1.
    \]
    Therefore, we apply Corollary~\ref{cor: tc_k}~\eqref{2second} to conclude that $\TC_k(SP^n(N_g))=2nk$ for $ k\ge\left\lfloor\tfrac{\gap_2(P^{2n})}{g-1}\right\rfloor+3$. For brevity, let $\D:=\left\lfloor\tfrac{\gap_2(P^{2n})}{g-1}\right\rfloor+3$. Then we can write
    \[
    f_{SP^n(N_g)}(t)=\sum_{k=1}^{\D-2}\ \TC_{k+1}(SP^n(N_{g})) \ t^k \ +\ 2n\sum_{k=\D-1}^{\infty} (k+1)t^k.
        \]
Of course, $\sum_{k=\D-1}^{\infty} (k+1)t^k=\tfrac{d}{dt}\left(\sum_{k=\D-1}^{\infty} t^{k+1}\right)=\tfrac{d}{dt}\left(\tfrac{t^{\D}}{1-t}\right)$. Hence,
\begin{multline*}
     f_{SP^n(N_g)}(t)=\sum_{k=1}^{\D-2}\ \TC_{k+1}(SP^n(N_{g})) \ t^k \ +\ 2n\ \frac{\D t^{\D-1}-(\D-1)t^{\D}}{(1-t)^2}
\\
=\frac{1}{(1-t)^2} \left(2n\left(\D t^{\D-1}-(\D-1)t^{\D}\right)\ + \ (1-t)^2\ \sum_{k=1}^{\D-2}\ \TC_{k+1}(SP^n(N_{g})) \ t^k \right).
\end{multline*}
Because $\cat(SP^n(N_g))=2n$ due to~\cite[Theorem~B]{Ja}, we see that the integral polynomial inside the above bracket indeed satisfies the asserted properties. This verifies the rationality conjecture in the general case with the upper bound $\D$ on the degree of $P_{SP^n(N_g)}(t)$.

In the case of (1), if $k\ge 3$ is odd, then
\[
g\ge\left\lfloor\frac{\gap_2(P^{2n})}{2}\right\rfloor+2 \ge \left\lfloor\frac{\gap_2(P^{2n})}{k-1}\right\rfloor+2\ge\left\lfloor\frac{\gap_k(P^{2n})}{k-1}\right\rfloor+2\ge\left\lceil \frac{\gap_k(P^{2n})}{k-1}\right\rceil+1.
\]
Similarly, if $k\ge 4$ is even, then the condition on $g$ implies
\[
g\ge\left\lfloor\frac{\gap_2(P^{2n})}{2}\right\rfloor+2 \ge \left\lfloor\frac{\gap_2(P^{2n})}{k-2}\right\rfloor+2\ge\left\lfloor\frac{\gap_k(P^{2n})}{k-2}\right\rfloor+2.
\]
Thus, Corollary~\ref{cor: tc_k}~\eqref{2second} gives $\TC_k(SP^n(N_g))=2nk$ for all $k\ge 3$. In the case of (2), we again get $\TC_k(SP^n(N_g))=2nk$ for all $k\ge 3$ by Corollary~\ref{cor: tc_k}~\eqref{2second}. So, in any case, we have the expression
\begin{align*}
 f_{SP^n(N_g)}(t)&=\TC_2(SP^n(N_g)) \ t\ + \ 2n\ \frac{d}{dt}\left(\sum_{k=2}^{\infty}t^{k+1} \right) 
\\
&=
\frac{1}{(1-t)^2}\left(\cat\left(SP^n(N_g)\right)\left(3t^2-2t^3\right) \ + \ \TC_2(SP^n(N_g)) \ t\left(1-t\right)^2\right).
\end{align*}
This completes the proof.
\end{proof}

The rationality conjecture, which is open for closed manifolds, has been verified for only a few classes of spaces (and even fewer classes of non-orientable manifolds), the most recent being that of the symmetric products of closed orientable surfaces, see~\cite{Ja} (and also~\cite[Section~8]{FO} and~\cite{HL} for some other known classes). Theorem~\ref{thm: rationality} extends the class of such (non-orientable) manifolds, for many of which the $\TC$-generating function can be explicitly computed.

\begin{remark}
    Fix some $n\ge 1$. Then it follows from the proof of Theorem~\ref{thm: rationality} that for $g\ge\left\lfloor \tfrac{\gap_2(P^{2n})}{2}\right\rfloor+2$ and $r\ge k\ge 3$, 
   \[
   \TC_r(SP^n(N_g))-\TC_k(SP^n(N_g))=(r-k)\cat(SP^n(N_g))=2n(r-k).
   \]
   In particular, the sequence $\{\TC_{k}(SP^n(N_g))\}_{k\ge 3}$ is an arithmetic progression with common difference $2n=\cat(SP^n(N_g))$ for $g$ moderately large. Of course, the same conclusions hold for $\cat(C_{\Delta_{k}SP^n(N_g)})$ as well in this range due to Corollary~\ref{cor: tc_k}~\eqref{2second}.
\end{remark}

We conclude this section by making an observation regarding the analog of the Ganea conjecture on LS-category~\cite{CLOT} for sequential topological complexity.

\begin{corollary}
    For a triplet $(n,g,k)$ satisfying the hypotheses of Corollary~\ref{cor: tc_k}~\eqref{2second}, we have that
    \[
    \TC_k(SP^n(N_g)\times S^m)=\TC_k(SP^n(N_g))+ \TC_k(S^m)
    \]
    holds for all odd $m\ge 1$. In particular, the $\TC_k$-Ganea conjecture holds for $SP^n(N_g)$ with respect to odd spheres.
\end{corollary}
\begin{proof}
For such triplets, we have $\zcl_k(SP^n(N_g))=\TC_k(SP^n(N_g))$. We also have from~\cite{Ru} that $\zcl_k(S^m)=\TC_k(S^m)=k-1$. The results of~\cite{BGRT} imply that $\zcl_k(SP^n(N_g))+\zcl_k(S^m)\le\zcl_k(SP^n(N_g)\times S^m)$ because $m$ is odd, and that $\TC_k(SP^n(N_g)\times S^m)\le\TC_k(SP^n(N_g))+\TC_k(S^m)$. This completes the proof.
\end{proof}

\section{Cohomological aspects}\label{sec: cohomo}

\subsection{The cohomology of $SP^n(N_g)$}\label{subsec: coho ring}

The description of the $\Z_2$-cohomology ring of $SP^n(N_g)$ is due to Kallel and Salvatore~\cite{KS}. We recall the details below.

The closed non-orientable surface $N_g$ (for $g\ge 1$) has a CW  structure consisting of one $0$-cell, $g$ $1$-cells, and one $2$-cell attached by the word $a_1^2a_2^2\cdots a_g^2$, where $a_i$ is the generator in $H_1(P_i^2)$ for $1\le i \le g$. Here, $P_i^2$ is a copy of the projective plane, and we use the connected sum expression $N_g=\#_{i=1}^gP_i^2$. The elements $a_i$ form a basis for $H_1(N_g)$. We use the same notations for their images in $H_1(SP^n(N_g))$ under the map induced by a basepoint inclusion $N_g \hookrightarrow SP^n(N_g)$. Let $b\in H_2(N_g)$ denote the $\Z_2$-fundamental class and its image in $H_2(SP^n(N_g))$. We denote the Hom duals of $a_i$ and $b$ by $x_i$ and $y$, respectively. In $H^*(N_g)$, we have $x_i^2=y$ and $x_ix_j=0$ for $i\ne j$.

\begin{thm}[\protect{\cite[Section 4]{KS}}]\label{thm: ks description}
For $n,g\ge 1$, the cohomology ring $H^*(SP^n(N_g))$ is generated by the degree $1$ classes $x_1,\ldots,x_g$ and the degree $2$ class $y$, subject only to the following relations:
\begin{enumerate}
\item $x_i^2=y$ for each $1\le i\le g$, and
\item $x_{i_1}\cdots x_{i_r}y^s=0$ whenever $r+s > n$ for distinct indices $i_1,\ldots,i_r$.
\end{enumerate}
\end{thm}

It is convenient to use $I_{n,g}$ to denote the ideal generated by the above two sets of relations. Then the theorem says that $H^*(SP^n(N_g))\cong\Z_2\left[x_1,x_2,\ldots,x_g,y\right]/I_{n,g}$.

\begin{remark}\label{rem: ks non-zero}
It follows from~\cite[Lemma 19]{KS} that a monomial $x_{i_1}\cdots x_{i_r}y^s$ in $H^*(SP^n(N_g))$ with distinct indices $i_1,\ldots,i_r$ vanishes if and only if $r+s>n$. 
\end{remark}

In fact, the following description can be derived from~\cite[Theorem~6]{KS}.

\begin{thm}\label{thm: additive basis}
For $n,g\ge 1$, an additive basis (the standard one for our purposes) for the vector space $H^*(SP^n(N_g))$ is formed by the monomials
\[
y^{t}\ \prod_{i=1}^gx_{i}^{\varepsilon_i} = x_1^{2t+\varepsilon_1}\ \prod_{i=2}^gx_{i}^{\varepsilon_i}
\]
satisfying $\varepsilon_i\in\{0,1\}$, $t\ge 0$, and $t+\sum_{i=1}^g\varepsilon_i\le n$, where we write $y=x_1^2$ without loss of generality (cf. Theorem~\ref{thm: ks description}).
\end{thm}

Taking $g=1$ (resp. $n=1$), Theorem~\ref{thm: additive basis} recovers the cohomology ring of $H^*(P^{2n})$ for $n\ge 1$ (resp. $H^*(N_g)$ for $g\ge 1$).

\begin{remark}\label{rem: relocatedcontent}
For $g'\ge g$, consider the ring homomorphism
\[
\iota\colon H^*(SP^n(N_g))\to H^*(SP^n(N_{g'}))
\]
given by $\iota(y)=y$ and $\iota(x_i)=x_i$ for $1\le i \le g$. Since $\iota$ maps the additive basis of its domain to a subset of the additive basis of its codomain, it is a monomorphism. We will need the ring monomorphism $H^*(SP^n(N_g))^{\otimes k}\to H^*(SP^n(N_{g'}))^{\otimes k}$ given by the $k$-fold tensor product of $\iota$.
\end{remark}

\subsection{The capacity function and the cohomology of products of $SP^n(N_g)$}\label{subsec: coho ring2}
For a given integer $k\ge 2$, we take the tensor product additive basis in the ring $H^*((SP^n(N_g))^k)\cong H^*(SP^n(N_g))^{\otimes k}$ coming from the basis in Theorem~\ref{thm: additive basis}. For $1\le r\le k$, let $\proj_r\colon (SP^n(N_g))^k \to SP^n(N_g)$ be the projection onto the $r$-th Cartesian factor. For $1\le i \le g$, let
\begin{equation}\label{defintionofxri}
x_{r,i}:=\proj_r^*(x_i) = 1\otimes\cdots\otimes 1 \otimes x_i\otimes 1\otimes\cdots\otimes 1 \in H^*(SP^n(N_g))^{\otimes k},
\end{equation}
where $x_i$ appears only at the $r$-th position.

\begin{definition}\label{def: capacity}
The \emph{capacity} of a basis element $x_1^ax_2^{\varepsilon_2}\cdots x_g^{\varepsilon_g}\in H^*(SP^n(N_g))$ in Theorem~\ref{thm: additive basis} is defined as
\[c\left(x_1^a\prod_{i=2}^gx_i^{\varepsilon_i}\right)=n-\left(\left\lfloor\frac{a+1}{2}\right\rfloor+\sum_{i=2}^g\varepsilon_i\right)\ge 0.
\]
We also set $c(0)=-\infty$. Similarly, the \emph{capacity} of a tensor product basis element $b_1\otimes\cdots\otimes b_k\in H^*(SP^n(N_g))^{\otimes k}$ is defined as
\[
c\left(b_1\otimes\cdots\otimes b_k\right)=\sum_{r=1}^kc(b_{r}),
\]
where $b_r:=\proj_r^*(x_{1}^{a_{r,1}}x_{2}^{\varepsilon_{r,2}}\cdots x_{g}^{\varepsilon_{r,g}})=x_{r,1}^{a_{r,1}}x_{r,2}^{\varepsilon_{r,2}}\cdots x_{r,g}^{\varepsilon_{r,g}}\in H^*(SP^n(N_g))^{\otimes k}$ for $1\le r\le k$.
\end{definition}

\begin{remark}\label{motivaciondelacapacidad}
For $a\geq0$ and $\varepsilon_2,\ldots,\varepsilon_g\in\{0,1\}$, Remark~\ref{rem: ks non-zero} and Theorem~\ref{thm: additive basis} imply that the monomial $\mu:=x_1^ax_2^{\varepsilon_2}\cdots x_g^{\varepsilon_g}\in H^*(SP^n(N_g))$ is either a basis element or trivial, depending on whether $c(\mu)\geq0$ or $c(\mu)=-\infty$. In particular, for each $j\in\{1,\ldots,g\}$, $$x_1^ax_2^{\varepsilon_2}\cdots x_j^{\varepsilon_j}\cdot x_{j+1}^{\varepsilon_{j+1}} \cdots x_g^{\varepsilon_g}=0 \iff c(x_1^ax_2^{\varepsilon_2}\cdots x_j^{\varepsilon_j})<\varepsilon_{j+1}+\cdots+\varepsilon_g.$$ In other words, $c(x_1^ax_2^{\varepsilon_2}\cdots x_j^{\varepsilon_j})$ is the maximal number of distinct elements in $\{x_{j+1},\ldots,x_g\}$ whose product with $x_1^ax_2^{\varepsilon_2}\cdots x_j^{\varepsilon_j}$ is non-zero. 
\end{remark}

The following result, a straightforward Cartesian-wise extension of the previous fact, will be crucial for the computations in the final section of the paper.

\begin{lemma}\label{lem: tool for capacity}
Let $g'>g$ and pick a basis element $b_1\otimes \cdots\otimes b_k\in H^*(SP^n(N_g))^{\otimes k}$ as in Definition~\ref{def: capacity}. In terms of the ring map $H^*(SP^n(N_g))^{\otimes k}\hookrightarrow H^*(SP^n(N_{g'}))^{\otimes k}$ in Remark~\ref{rem: relocatedcontent}, we have
\[
c\left(b_1\otimes\cdots\otimes b_k\cdot\prod_{t=1}^jx_{e_t,i_t}\right)\le c\left(b_1\otimes\cdots\otimes b_k\right)-j
\]
for any integers $e_t\in\{1,\ldots,k\}$ and $i_t\in\{g+1,\ldots,g'\}$ satisfying $i_t\ne i_s$ whenever $t\ne s$ and $e_t=e_s$. In particular, $b_1\otimes\cdots\otimes b_k \cdot   \prod_{t=1}^jx_{e_t,i_t}$ vanishes whenever $c(b_1\otimes\cdots\otimes b_k)<j$.
\end{lemma}

We close the section with a couple of additional properties of the capacity function that will be useful later in the paper.

\begin{proposition}\label{prop: prop y from notes}
Let $p_{r}\in\{0,\ldots, 2n\}$ for $1\le r \le k$. If $\sum_{r=1}^kp_r\le A$, then
\[
c\left(x_{1,1}^{2n-p_1}\otimes\cdots\otimes x_{k,1}^{2n-p_k}\right)\le\left\lfloor\frac{A}{2}\right\rfloor.
\]
\end{proposition}
\begin{proof}
By definition, $c(x_{r,1}^{2n-p_r})=n-\left\lfloor\tfrac{2n-p_r+1}{2}\right\rfloor=\left\lfloor\tfrac{p_r}{2}\right\rfloor$. Therefore,
\[
c\left(x_{1,1}^{2n-p_1}\otimes\cdots\otimes x_{k,1}^{2n-p_k}\right)=\sum_{r=1}^k \left\lfloor\frac{p_r}{2}\right\rfloor\le\left\lfloor\frac{\sum_{r=1}^kp_r}{2}\right\rfloor\le\left\lfloor\frac{A}{2}\right\rfloor.
\]
\end{proof}

\begin{proposition}\label{prop: prop z from notes}
Let $b_1\otimes\cdots\otimes b_k\in H^*(SP^n(N_g))^{\otimes k}\hookrightarrow H^*(SP^n(N_{g+1}))^{\otimes k}$ be a basis element. Then for any basis element $\alpha$ in the expansion of a product
\[
p=b_1\otimes\cdots\otimes b_k\cdot\prod_{r=2}^k\left(x_{1,g+1}+x_{r,g+1}\right)^{\varepsilon_r},
\]
we have the inequality
\[
c(\alpha)\le c\left(b_1\otimes\cdots\otimes b_k\right)-\left\lfloor\frac{1+\sum_{r=2}^k\varepsilon_r}{2}\right\rfloor.
\]
In particular, the product $p$ vanishes if $c(b_1\otimes\cdots\otimes b_k)<\left\lfloor\tfrac{1+\sum_{r=2}^k\varepsilon_r}{2}\right\rfloor$.
\end{proposition}

\begin{proof}
For convenience, write $\sum_{r=2}^k\varepsilon_r=2q+\varepsilon$, where $\varepsilon\in\{0,1\}$. By Lemma~\ref{lem: tool for capacity}, the summand in the expansion of $p$ with the largest capacity is $b_1\otimes\cdots\otimes b_k\cdot x_{1,g+1}^{2q+\varepsilon}$, which is the same as $(b_1\ x_{1,1}^{2q}\ x_{1,g+1}^{\varepsilon})\otimes b_2\otimes\cdots\otimes b_k$. We have for its capacity that
\begin{align*}
c\left((b_1 \ x_{1,1}^{2q}\ x_{1,g+1}^{\varepsilon})\otimes b_2\otimes\cdots\otimes b_k\right)
&
\le c\left(b_1\otimes\cdots\otimes b_k\right)-(q+\varepsilon)
\\
&
=c\left(b_1\otimes\cdots\otimes b_k\right)-\left\lfloor\frac{1+2q+\varepsilon}{2}\right\rfloor.
\end{align*}
\end{proof}

\section{$k$-th zero-divisors}\label{sec: zero-divisor}

Let $n,g\ge 1$ and $k\ge 2$ be integers. Let $\Delta_{k,n,g}\colon SP^n(N_g)\to (SP^n(N_g))^k$ be the diagonal map. We now study the kernel of the $k$-fold iterated cup product multiplication, i.e., the kernel of the homomorphism $$\Delta^*_{k,n,g}\colon H^*(SP^n(N_g))^{\otimes k}\cong H^*((SP^n(N_g))^k)\to H^*(SP^n(N_g))$$ induced by $\Delta_{k,n,g}$ in $\Z_2$-cohomology. The subring $\Ker(\Delta_{k,n,g}^*)$ is the ideal of the $k$-th zero-divisors of $SP^n(N_g)$. For instance, in terms of the notation in~\eqref{defintionofxri}, elements of the form $x_{r,i}+x_{s,i}$ are $k$-th zero divisors.

The case $g=1$ of the following result was shown in~\cite{CAG+}.

\begin{proposition}\label{prop: ideal description}
For $k\ge 2$ and $n,g\ge 1$, the $k$-th zero-divisors $x_{1,i}+x_{r,i}$, where $2\le r\le k$ and $1\le i \le g$, generate the ideal $\Ker(\Delta^*_{k,n,g})$ of $k$-th zero-divisors.
\end{proposition}

In preparation for the proof, we first define the ideal
\[
J_{k,n,g}:=\left\langle x_{1,i}+x_{r,i}\ \middle|\ 2\le r\le k \ \text{and}\ 1\le i \le g\right\rangle
\]
generated by the specific zero-divisors of $SP^n(N_g)$ mentioned in Proposition~\ref{prop: ideal description}. It is clear that $J_{k,n,g}$ is a sub-ideal of $\Ker(\Delta_{k,n,g}^*)$. Additionally, for $t\ge 0$ and $\varepsilon_{2},\ldots,\varepsilon_g\in \{0,1\}$, we set
\[
m\left(t,\varepsilon_{2},\ldots, \varepsilon_{g}\right):= x_{1}^{t} \ x_{2}^{\varepsilon_2} \ \cdots  \ x_{g}^{\varepsilon_{g}}\in H^*(SP^n(N_g)),
\]
and for $1\le r\le k$,
\[
m_r(t,\varepsilon_{2},\ldots, \varepsilon_{g}):= \proj_r^*(m\left(t,\varepsilon_{2},\ldots, \varepsilon_{g}\right))=x_{r,1}^{t} \ x_{r,2}^{\varepsilon_2} \ \cdots  \ x_{r,g}^{\varepsilon_{g}}\in H^*(SP^n(N_g))^{\otimes k},
\]
where $\proj_r$ is as in~\eqref{defintionofxri}.

\begin{proof}[Proof of Proposition~\ref{prop: ideal description}]
Consider the commutative diagram
\[
\begin{tikzcd}
H^*(SP^n(N_g)) \arrow{r}{\proj_1^*} \arrow{dr}{i}
&
H^*(SP^k(N_g))^{\otimes k}   \arrow{r}{\Delta^*_{k,n,g}}  \arrow{d}{q}
&
H^*(SP^n(N_g))
\\
&
H^*(SP^k(N_g))^{\otimes k}/J_{k,n,g}, \arrow{ur}{j}
&
\end{tikzcd}
\]
where $q$ is the canonical projection and $j$ is determined by $\Delta_{k,n,g}^*$ on the quotient since $J_{k,n,g}\subset\Ker(\Delta_{k,n,g}^*)$. The composition on the top is an isomorphism since $\text{proj}_1 \circ \Delta_{k,n,g}$ is the identity on $SP^n(N_g)$. We prove that $i=q\circ \proj_1^*$ is an isomorphism, which forces the corresponding property for $j$ and, thus, the desired equality $J_{k,n,g}=\Ker(\Delta_{k,n,g}^*)$.

Note that $i$ is a monomorphism since $j\circ i$ is so. On the other hand, the congruence $x_{r,i}\equiv x_{1,i}\text{ mod }(J_{k,n,g})$, valid for $1\le r\le k$ and $1\le i \le g$, shows
\[
\prod_{r=1}^k
m_r\left( t_r,\varepsilon_{r,2},\ldots,\varepsilon_{r,g}\right) \equiv m_1\left(\sum_{r=1}^k t_r,\sum_{r=1}^k \varepsilon_{r,2},\ldots,\sum_{r=1}^k \varepsilon_{r,g}\right) \text{ mod}\left(J_{k,n,g}\right).
\]
The element on the left-hand side is a typical generator of the domain of $q$, while the image of the element under $q$ on the right-hand side belongs to $\Im(i)$. It follows that $i$ is an epimorphism.
\end{proof}

Because $\zcl_k(SP^n(N_g))$ is the length of the longest non-trivial cup product in $\Ker(\Delta^*_{k,n,g})$, Proposition~\ref{prop: ideal description} yields the following characterization of $\zcl_k(SP^n(N_g))$.

\begin{corollary}\label{cor: zcl simplified}
For $k\ge 2$ and $n,g\ge 1$, we have that
\[
\zcl_k(SP^n(N_g))=\max\left\{\sum_{\substack{2\le r\le k \\ 1\le i \le g}}a_{r,i} \ \middle| \ 0\ne\prod_{\substack{2\le r\le k \\ 1\le i \le g}}\left(x_{1,i}+x_{r,i}\right)^{a_{r,i}}\in H^*(SP^n(N_g))^{\otimes k} \right\}.
\]
\end{corollary}

For fixed $n$, we now compare the $\zcl_k(SP^n(N_g))$ values as $k$ and $g$ vary.

\begin{corollary}\label{cor: zcl with g changing}
For $k\ge 2$ and $n,g\ge 1$, we have that
\[
\cdots\ge \zcl_k(SP^n(N_{g+1}))\ge \zcl_k(SP^n(N_{g})) \ge \cdots\ge\zcl_k(P^{2n}).
\]
\end{corollary}

\begin{proof}
In view of Proposition~\ref{prop: ideal description}, this follows from the fact that the ring inclusion $$H^*(SP^n(N_{g}))^{\otimes k}\hookrightarrow H^*(SP^n(N_{g'}))^{\otimes k}$$ constructed for $g'\geq g$ in Remark~\ref{rem: relocatedcontent} maps $x_{1,i}+x_{r,i}\in H^*(SP^n(N_{g}))^{\otimes k}$ into the corresponding $k$-th zero divisor in $H^*(SP^n(N_{g'}))^{\otimes k}$.
\end{proof}

\begin{corollary}\label{cor: zcl with k changing}
For $k\ge 2$ and $n,g \ge 1$, we have that
\[
\cdots\ge \zcl_{k+1}(SP^n(N_{g}))\ge \zcl_k(SP^n(N_{g})) \ge \cdots\ge\zcl_2(SP^{n}(N_g)).
\]
\end{corollary}

\begin{proof}
For $k_1\le k_2$, consider the projection $\pi\colon (SP^n(N_g))^{k_2} \to (SP^n(N_g))^{k_1}$ that forgets the last $k_2-k_1$ Cartesian factors of $(SP^n(N_g))^{k_2}$. This induces a monomorphism $\pi^*\colon H^*(SP^n(N_g))^{\otimes k_1}\to H^*(SP^n(N_g))^{\otimes k_2}$ satisfying $\pi^*(x_{r,i})=x_{r,i}$ for $1\le r\le k_1$ and $1\le i \le g$. Thus, the inequality $\zcl_{k_2}(SP^n(N_{g}))\ge \zcl_{k_1}(SP^n(N_{g}))$ holds in view of Proposition~\ref{prop: ideal description}.
\end{proof}

\begin{ex}\label{maximalidaddeTC}
Corollary~\ref{cor: zcl with g changing} gives $\TC_k(SP^n(N_g))=\zcl_k(SP^n(N_g))=2nk$ for any $g\ge 1$, provided $\zcl_k(P^{2n})=2nk$. In turn, as proved in ~\cite[Proposition~1.4]{Dav}, the latter equality holds for $k\ge 2^{\ell+1}-1$, where $\ell$ is the length of the longest string of consecutive $1$'s in the binary expansion of $n$. Note that $\TC_k(SP^n(N_g))$ can attain its maximal possible value $2nk$ without $\zcl_k(P^{2n})$ doing so, see Corollary~\ref{cor: tc_k}.
\end{ex}

\begin{remark}\label{rem: on different zcls}
The non-decreasing sequence in Corollary~\ref{cor: zcl with g changing} is bounded from above by $2nk$, so it is eventually constant. On the other hand, the sequence in Corollary~\ref{cor: zcl with k changing} is unbounded with an eventually strict linear behavior. Indeed, Example~\ref{maximalidaddeTC} shows that all terms with $k$ large enough in the sequence of Corollary~\ref{cor: zcl with k changing} are maximal possible (i.e., equal to $2nk$).
\end{remark}

The following result reduces the search for tuples $(a_{r,i})$ realizing $\zcl_k(SP^n(N_g))$ in Corollary~\ref{cor: zcl simplified}.

\begin{proposition}\label{prop: zcl search reduced}
For $k\ge 2$ and $n,g\ge 1$, the maximal sum in the set of Corollary~\ref{cor: zcl simplified} is attained with a tuple of exponents $(a_{r,i})$ satisfying the following three properties:
\begin{enumerate}
\item $a_{2,1}\ge a_{3,1}\ge\cdots\ge a_{k,1}$;\label{1}
\item $1\ge a_{2,2}\ge a_{2,3} \ge \cdots\ge a_{2,g}\ge 0$;\label{2}
\item $a_{r,i}\in\{0,1\}$ for all $2\le r\le k$ and $2\le i \le g$.\label{3}
\end{enumerate}
\end{proposition}

\begin{proof}
Assume $\zcl_k(SP^n(N_g))=\sum a_{r,i}$, as in Corollary~\ref{cor: zcl simplified}. Say $a_{r,i}=2\alpha_{r,i}+\varepsilon_{r,i}$ with $\varepsilon_{r,i}\in \{0,1\}$ and $\alpha_{r,i}\ge 0$ for $2\le r\le k$ and $2\le i \le g$. The relation $x_{r,i}^2=x_{r,1}^2$ gives
\begin{align*}\left(x_{1,i}+x_{r,i}\right)^{a_{r,i}}&= \left(x_{1,i}+x_{r,i}\right)^{2\alpha_{r,i}} \left(x_{1,i}+x_{r,i}\right)^{\varepsilon_{r,i}} \\ & = \left(x_{1,1}+x_{r,1}\right)^{2\alpha_{r,i}} \left(x_{1,i}+x_{r,i}\right)^{\varepsilon_{r,i} }.
\end{align*}
The non-zero product corresponding to the tuple $(a_{r,i})$ then becomes
\begin{align*}
\prod_{\substack{2\le r\le k \\ 1\le i \le g}}\left(x_{1,i}+x_{r,i}\right)^{a_{r,i}}
& =\prod_{r=2}^k\left(x_{1,1}+x_{r,1}\right)^{a_{r,1}} \prod_{\substack{2\le r\le k \\ 2\le i \le g}}\left(x_{1,1}+x_{r,1}\right)^{2\alpha_{r,i}} \prod_{\substack{2\le r\le k \\ 2\le i \le g}}\left(x_{1,i}+x_{r,i}\right)^{\varepsilon_{r,i}}
\\
& =\prod_{r=2}^k\left(x_{1,1}+x_{r,1}\right)^{a_{r,1}+2\sum_{i=2}^g\alpha_{r,i}} \prod_{\substack{2\le r\le k \\ 2\le i \le g}}\left(x_{1,i}+x_{r,i}\right)^{\varepsilon_{r,i}}
\\
& =:\prod_{\substack{2\le r\le k \\ 1\le i \le g}}\left(x_{1,i}+x_{r,i}\right)^{b_{r,i}}.
\end{align*}
The new set of exponents $(b_{r,i})$ satisfies Property~\eqref{3} and still realizes $\zcl_k(SP^n(N_g))$, since $\sum a_{r,i}=\sum b_{r,i}$ by construction. Next, a permutation $\sigma\in\Sigma_k$ produces an automorphism on $(SP^n(N_g))^k$ that permutes the Cartesian factors. If $\sigma$ fixes $1$, then the corresponding automorphism $\sigma^*$ on $H^*(SP^n(N_g))^{\otimes k}$ satisfies
\[
0\ne \prod_{\substack{2\le r\le k \\ 1\le i \le g}}\left(x_{1,i}+x_{r,i}\right)^{b_{r,i}} \xmapsto{\sigma^*}  \prod_{\substack{2\le r\le k \\ 1\le i \le g}}\left(x_{1,i}+x_{\sigma(r),i}\right)^{b_{r,i}}= \prod_{\substack{2\le r\le k \\ 1\le i \le g}}\left(x_{1,i}+x_{r,i}\right)^{b_{\sigma^{-1}(r),i}} .
\]
Thus, the tuple of exponents $(c_{r,i}:=b_{\sigma^{-1}(r),i})$ realizes $\zcl_k(SP^n(N_g))$ and satisfies both Properties~\eqref{1} and~\eqref{3} for suitably chosen $\sigma\in\Sigma_k$.

Lastly, any permutation $\rho\in\Sigma_g$ induces an automorphism $\rho^*$ on $\Z_2[x_1,\ldots,x_g,y]$ by permuting the classes $x_i$. This descends to an automorphism, still denoted by $\rho^*$, on $H^*(SP^n(N_g))\cong\Z_2[x_1,\ldots,x_g,y]/I_{n,g}$, where $I_{n,g}$ is the ideal of relations described in Theorem~\ref{thm: ks description}. Choosing $\rho$ so that $c_{2,\rho^{-1}(2)}\ge\cdots\ge c_{2,\rho^{-1}(g)}$ and $\rho(1)=1$, we get 
\[
\prod_{\substack{2\le r\le k \\ 1\le i \le g}}\left(x_{1,i}+x_{r,i}\right)^{c_{r,i}} \xmapsto{(\rho^*)^{\otimes k}} \prod_{\substack{2\le r\le k \\ 1\le i \le g}}\left(x_{1,\rho(i)}+x_{r,\rho(i)}\right)^{c_{r,i}}=\prod_{\substack{2\le r\le k \\ 1\le i \le g}}\left(x_{1,i}+x_{r,i}\right)^{c_{r,\rho^{-1}(i)}}.
\]
So, the tuple of exponents $(d_{r,i}:=c_{r,\rho^{-1}(i)})$ attains $\zcl_k(SP^n(N_g))$ and satisfies Properties~\eqref{1},~\eqref{2}, and~\eqref{3}.
\end{proof}

\section{The case of $k=2$}\label{sec: k=2}

In this section, we prove Theorem~\ref{thm: zcl k=2}.
Recall that $2^e\le n<2^{e+1}$ for some $e\ge 0$. For convenience, let us write $n=2^e+d$, where $0\le d<2^e$. We divide the proof into two cases.

\begin{proof}[\textbf{Case (1): $g\le 2d+1$}]
By Proposition~\ref{prop: zcl search reduced}, the largest possible product of zero-divisors that is potentially non-zero is
\begin{equation}\label{eq: monomial 1 for zcl_2}
\left(x_{1,1}+x_{2,1}\right)^{2^{e+2}-1}\ \prod_{i=2}^g \left(x_{1,i}+x_{2,i}\right).
\end{equation}
We show~\eqref{eq: monomial 1 for zcl_2} is indeed non-zero, which yields the assertion. Note that, although all binomial coefficients $2^{e+2}-1 \choose j$ with $0\leq j\leq 2^{e+2}-1$ are odd (by Lucas's theorem), Theorems~\ref{thm: ks description} and~\ref{thm: additive basis} imply that the expansion of~\eqref{eq: monomial 1 for zcl_2} as a sum of basis elements contains the monomial 
\[
\left(x_{1,1}^{2^{e+2}-1-2n}\otimes x_{2,1}^{2n} \,\right)\cdot \prod_{i=2}^g x_{1,i}
=\left(x_{1,1}^{2^{e+1}-2d-1}\ \prod_{i=2}^g x_{1,i}\right)\otimes x_{2,1}^{2n} \ ,  
\]
which is indeed a basis element in view of Theorem~\ref{thm: additive basis} and the current hypothesis $g\leq 2d+1$.
\end{proof}

\begin{proof}[\textbf{Case (2): $g> 2d+1$}]
We have established the non-vanishing of the cup product
\[
\left(x_{1,1}+x_{2,1}\right)^{2^{e+2}-1}\ \prod_{i=2}^{2d+1} \left(x_{1,i}+x_{2,i}\right), \]
so that $\zcl_2(SP^n(N_g))\ge 2^{e+2}+2d-1$. Equality holds here in view of Proposition~\ref{prop: zcl search reduced}~\eqref{2} and the fact that
\[
\left(x_{1,1}+x_{2,1}\right)^{2^{e+2}-1-\ell}\ \prod_{i=2}^{2d+\ell+2} \left(x_{1,i}+x_{2,i}\right)   
\]
vanishes for all $0\le\ell\le\min\{2^{e+2}-1,g-2d-2\}$. Indeed, a potential basis summand in the expansion of such a product would have the form
\[
    x_{1,1}^a\otimes  x_{2,1}^b \ \cdot\ \prod_{i\in C} x_{1,i}\otimes \prod_{j\in D}x_{2,j}= \left(x_{1,1}^a\prod_{i\in C} x_{1,i}\right) \otimes \left(x_{2,1}^b \prod_{j\in D}x_{2,j}\right),
\]
where $a+b=2^{e+2}-1-\ell$ and $C\cup D=\{2,\ldots,2d+\ell+2\}$ with $C\cap D=\varnothing$. But then
\[
\max\left\{\left\lfloor \frac{a+1}{2} \right\rfloor+|C|, \left\lfloor \frac{b+1}{2} \right\rfloor+|D| \right\} \le n,
\] 
so that
\begin{equation}\label{eq: submonomial 4}
\left\lfloor \frac{a+1}{2} \right\rfloor + \left\lfloor \frac{b+1}{2} \right\rfloor + 2d+\ell+1\le 2^{e+1}+2d,    
\end{equation}
which is impossible. This is because if $\ell$ is even, then $a$ and $b$ would have opposite parities and so
\[
\left\lfloor \frac{a+1}{2} \right\rfloor + \left\lfloor \frac{b+1}{2} \right\rfloor +\ell+1=\left\lfloor \frac{a+b+2}{2} \right\rfloor +\ell+1 = 2^{e+1}+\frac{\ell}{2}+1,
\]
which contradicts~\eqref{eq: submonomial 4}. Likewise, if $\ell$ is odd, then $a$ and $b$ would have the same parity, so that
\[
\left\lfloor \frac{a+1}{2} \right\rfloor + \left\lfloor \frac{b+1}{2} \right\rfloor +\ell+1\ge \left\lfloor \frac{a+b+2}{2} \right\rfloor -1 + \ell +1= 2^{e+1}+\frac{\ell+1}{2},
\]
which again contradicts~\eqref{eq: submonomial 4}.
\end{proof}

\begin{remark}\label{rem: bad zero divisors}
    The choice of zero-divisors of $SP^n(N_g)$ matters while forming long lengths of non-trivial cup products in $H^*(SP^n(N_g))^{\otimes 2}$. Indeed, it was explained in~\cite[Section 6.C]{Ja} that $\prod_{i=1}^j(x_{1,i}+x_{2,i})^2=0$ for $j\le\min\{g,n\}$. 
\end{remark}

\section{The case of odd $k\ge 3$}\label{sec: k odd}

In this section, we first develop a method to approach Theorem~\ref{thm: gap k ge 3} and then obtain a proof of its part~\eqref{first} for $k\ge 3$ odd. 

\subsection{General preparation}\label{subsec: prep}
The discussion in this subsection applies to all $k\ge 3$ (and not just to odd $k$), and will be used in Section~\ref{sec: k even} as well.

Working on $H^*(P^{2n})^{\otimes k}\cong H^*(SP^n(N_1))^{\otimes k}$, let $\mathcal{L}_{n,k}$ denote the set consisting of all those non-negative integers $\ell\leq2n$ for which there is a non-zero cup product as in~\eqref{eq: assume} below with the indicated basis element in its expansion.
\begin{equation}\label{eq: assume}
\prod_{r=2}^k\left(x_{1,1}+x_{r,1}\right)^{a_{r,1}}
=x_{1,1}^{\ell}\otimes x_{2,1}^{2n} \otimes\cdots\otimes x_{k,1}^{2n}+\cdots.    
\end{equation}
Note that other basis elements (if any) in the expansion of~\eqref{eq: assume} have the form
\begin{equation}\label{laforma}
x_{1,1}^{t_1}\otimes\cdots\otimes x_{k,1}^{t_k} \text{ with } t_r<2n \text{ for some }r\in\{2,\ldots, k\}.
\end{equation} 
Basis elements as the one indicated in~\eqref{eq: assume} play a critical role in~\cite{Dav}, where it is shown that $\zcl_k(P^{2n})= a_{2,1}+a_{3,1}+\cdots+a_{k,1}=2n(k-1)+\ell$ and
\begin{equation}\label{probadoporDon}
2n-\ell= \gap_k(P^{2n})=\max_{i\in S(2n)}\{0,2^{i+1}-1-kZ_i(2n)\},    
\end{equation}
provided $\ell=\max\mathcal{L}_{n,k}$; see~\eqref{eq: davis gap main} and Definition~\ref{defn: notation}.

Building on Davis's work, we now consider \emph{extended} products
\begin{equation}\label{productoextendido}
     \prod_{\substack{2\le r\le k \\ 1\le i \le g}}\!\!\left(x_{1,i}+x_{r,i}\right)^{a_{r,i}}     =\left(x_{1,1}^{\ell}\otimes x_{2,1}^{2n}\otimes \cdots\otimes x_{k,1}^{2n}+\cdots\right)\!\prod_{\substack{2\le r\le k \\ 2\le i \le g}}\!\!\left(x_{1,i}+x_{r,i}\right)^{a_{r,i}}
\end{equation}
for each $\ell\in \mathcal{L}_{n,k}$ with corresponding product~\eqref{eq: assume}, which fixes the exponents $a_{r,1}$, and for each set of additional exponents $a_{r,i}\in \{0,1\}$ with $2\le r\le k$ and $2\le i \le g$. This time, we focus on the monomial
\begin{equation}\label{elmonomio}
    \left(x_{1,1}^\ell \prod_{i=2}^g x_{1,i}^{\sum_{r=2}^{k}a_{r,i}}\right)\otimes x_{2,1}^{2n}\otimes\cdots\otimes x_{k,1}^{2n},
\end{equation}
which, because of~\eqref{laforma}, arises in the expansion of the left-hand side of~\eqref{productoextendido} only from the multiplication of the leading summands of factors in the product on the right-hand side of~\eqref{productoextendido}. In particular, assuming~(\ref{elmonomio}) is non-trivial, we get
\begin{equation}\label{eq: imp1}
  \sum_{\substack{2\le r\le k \\ 1\le i \le g}}a_{r,i} \le\zcl_k(SP^n(N_g)) \quad\text{and}\quad 
  \gap_k(SP^n(N_g))  \le 2nk-\sum_{\substack{2\le r\le k \\ 1\le i \le g}}a_{r,i}.
  \end{equation}
Furthermore, as discussed in~\eqref{probadoporDon}, the last two inequalities specialize to
\begin{eqnarray}\label{eq: imp2}
  \begin{split}
\zcl_k(P^{2n})+\sum_{\substack{2\le r\le k \\ 2\le i \le g}}a_{r,i}\le\zcl_k(SP^n(N_g));
\\
\gap_k(SP^n(N_g))\le \gap_k(P^{2n})-\sum_{\substack{2\le r\le k \\ 2\le i \le g}}a_{r,i},
\end{split}
\end{eqnarray}
provided $\ell=\max\mathcal{L}_{n,k}$. The benefit of the last two estimates is that the quantities $\zcl_k(P^{2n})$ and $\gap_k(P^{2n})$ appearing therein are completely understood (see~\eqref{probadoporDon}). Our target quantities $\zcl_k(SP^n(N_g))$ and $\gap_k(SP^n(N_g))$ can then be determined in a number of cases by proving that some suitable choice of the additional exponents $a_{r,i}$ ($i\geq2$) renders equalities in~\eqref{eq: imp2}. The simplest such scenario is given in Proposition~\ref{prop: main tool} below. In the few cases where this strategy falls short, the precise values of $\zcl_k(SP^n(N_g))$ and $\gap_k(SP^n(N_g))$ will be obtained using the more general setting in~\eqref{eq: imp1}.

\begin{proposition}\label{prop: main tool}
    Take $n\geq1$, $k\geq3$, and $\ell=\max\mathcal{L}_{n,k}$. In the setting of~\eqref{eq: imp2}, assume that the element in~\eqref{elmonomio} is non-zero and that
   $a_{r,i}=1$ for all $2\le r\le k$ and $2\le i \le g$. Then both estimates in~\eqref{eq: imp2} are sharp. In particular,
    \[
    \gap_k(SP^n(N_g)) = \max_{i\in S(2n)}\left\{0, 2^{i+1}-1-kZ_i(2n)-(k-1)(g-1) \right\}.
    \]
\end{proposition}

\begin{proof}
    The proof is parallel to that for the case $k=2$ and $g\leq2d+1$ in Section~\ref{sec: k=2}. Namely, under the present hypotheses, Proposition~\ref{prop: zcl search reduced} shows that~\eqref{productoextendido} is the largest possible non-zero product of $k$-th zero divisors. The fact that this product is non-zero comes from the assumed non-triviality of~\eqref{elmonomio}.
\end{proof}

\subsection{Proof of Theorem~\ref{thm: gap k ge 3}~\eqref{first} for $k\ge 3$ odd.}\label{subsec: odd k proof}
    In the expression~\eqref{eq: davis gap main} for $\gap_k(P^{2n})$, $Z_i(2n)$ is odd for each $i\in S(2n)$. Since $k$ is odd, $\gap_k(P^{2n})$ is even, say equal to $2m$. If $m=0$, then we are done since $\gap_k(SP^n(N_g))=0$ in view of Corollary~\ref{cor: zcl with g changing}.    So, let us assume that $m>0$. Take  $\ell:=\max\mathcal{L}_{n,k}$, so that $\ell=2(n-m)$ in~\eqref{probadoporDon}. Say $k=2p+1$ for some $p\ge 1$, and write $m=qp+s$ for some $s\in\{0,\ldots,p-1\}$ and $q\ge 0$. 

We divide the proof into two cases.

\begin{proof}[\textbf{Case (1): $1\le g\le q+1$}]
    In this case, Theorem~\ref{thm: ks description} and Remark~\ref{rem: ks non-zero} imply
    \[
     x_{1,1}^{\ell} \ \prod_{i=2}^gx_{1,i}^{k-1}=x_{1,1}^{2(n-m)+2p(g-1)} \ne 0,
    \]
because $x_{1,1}^{2}=\proj_1^*(y)$ and $n-m+p(g-1) \le n-m+pq=n-s\le n$. This yields the non-triviality of~\eqref{elmonomio} when $a_{r,i}=1$ for all $2\leq r\leq k$ and $2\leq i\leq g$. Proposition~\ref{prop: main tool} then yields the asserted description of $\gap_k(SP^n(N_g))$.
\end{proof}

\begin{proof}[\textbf{Case (2): $g\ge q+2$}]
Note that $2s<2p=k-1$ by definition. First, for $g=q+2$, we have
\[
x_{1,1}^{\ell}\ \prod_{i=2}^{q+1}x_{1,i}^{k-1} \ x_{1,q+2}^{2s} = x_{1,1}^{2(n-m)+2(pq+s)}=x_{1,1}^{2n}\ne 0,
\]
which yields the non-triviality of~\eqref{elmonomio} when $a_{r,i}=1$ for all relevant pairs $(r,i)$, except for $a_{r,g}=0$ with $2s+2\leq r\leq k$. Using~\eqref{eq: imp2}, we then get the optimal estimate
\[
\gap_k(SP^n(N_g))\le\gap_k(P^{2n})-\sum_{\substack{2\le r\le k\\ 2\le i \le g}}a_{r,i}=2m-(k-1)(g-2)-2s=0.
\]
For $g\ge q+3$, Corollary~\ref{cor: zcl with g changing} then gives $\gap_k(SP^n(N_g))\le\gap_k(SP^n(N_{q+2}))=0$. We now note that for each $g\ge q+2$, this is indeed the desired conclusion. This is because the inequalities $\gap_k(P^{2n})=2m<2pq+2p\le(k-1)(g-1)$ imply
\[
2^{i+1}-1-kZ_i(2n)-(k-1)(g-1)\le 2m-(k-1)(g-1) < 0
\]
for all $i\in S(2n)$, so $\max_{i\in S(2n)}\left\{0, 2^{i+1}-1-kZ_i(2n)-(k-1)(g-1)\right\}=0$.
\end{proof}

\begin{remark}
    In the above proof, the second hypothesis of Proposition~\ref{prop: main tool} is not satisfied in the case $g\ge q+2$, yet the estimates in~\eqref{eq: imp2} are sharp. Thus, the conditions in Proposition~\ref{prop: main tool} for the sharpness of the estimates in~\eqref{eq: imp2} are sufficient but not necessary.
\end{remark}

\section{The case of even $k\ge 4$}\label{sec: k even}

In this section, we complete the proof of the remaining parts of Theorem~\ref{thm: gap k ge 3}. We thus assume $k=2\lambda$ for some $\lambda\ge 2$. We can safely assume $\gap_k(P^{2n})>0$ too, because otherwise we are done in view of Corollary~\ref{cor: zcl with g changing}. Under these conditions,~\eqref{eq: davis gap main} implies that $\gap_k(P^{2n})$ is odd, say $\gap_k(P^{2n})=2h-1$ with $h\geq1$. We will analyze three different cases depending on how $h$ compares to the products $(g-1)(\lambda-1)$ and $(g-1)\lambda$. Indeed, in terms of the current notations, the conditions in Theorem~\ref{thm: gap k ge 3} translate as
\begin{align*}
g\le\left\lfloor\frac{\gap_k(P^{2n})}{k}\right\rfloor+1&\iff(g-1)\lambda<h; \\  g\ge\left\lfloor\frac{\gap_k(P^{2n})}{k-2}\right\rfloor+2&\iff(g-1)(\lambda-1)\ge h.
\end{align*} 
First, we prove Theorem~\ref{thm: gap k ge 3}~\eqref{first} for $k\ge 4$ even. 

\begin{proof}[\textbf{Case (1): $(g-1)\lambda< h$}]  
Taking $\ell:=\max\mathcal{L}_{n,k}$ and $a_{r,i}:=1$ for $2\leq r\leq k$ and $2\leq i\leq g$, the monomial
    \[
   x_{1,1}^{\ell} \ \prod_{i=2}^gx_{1,i}^{\sum_{r=2}^{k}a_{r,i}}= x_{1,1}^{2n-2h+1}\prod_{i=2}^g x_{1,i}^{k-1}=x_{1,1}^{2n-2h+ 2(g-1)(\lambda-1)}\prod_{i=1}^g x_{1,i}
    \]
is non-zero because $n-h+(g-1)(\lambda-1)+g\le n$. This yields the non-triviality of~\eqref{elmonomio} for the indicated values of $a_{r,i}$, and then Proposition~\ref{prop: main tool} gives the desired expression for $\gap_k(SP^n(N_g))$.  
\end{proof}
In preparation for the discussion of the other two cases, we fix
\begin{equation}\label{eq: for even k1}
\prod_{r=2}^k\left(x_{1,1}+x_{r,1}\right)^{a_{r,1}} = x_{1,1}^{\ell}\otimes x_{2,1}^{2n} \otimes\cdots\otimes x_{k,1}^{2n}+\cdots,
    \end{equation}
    which is a given non-zero product~\eqref{eq: assume} associated with $\ell:=\max\mathcal{L}_{n,k}$. Note that $\ell=2n-\gap_k(P^{2n})=2n-2h+1$ is odd, so that one of the exponents $a_{r_0,1}$ in~\eqref{eq: for even k1} must be odd. Setting $a'_{r_0,1}:=a_{r_0,1}-1$ and $a'_{r,1}:=a_{r,1}$ for $r\neq r_0$, Lucas's Theorem gives
    \begin{equation}\label{eq: for even k2}
\prod_{r=2}^k\left(x_{1,1}+x_{r,1}\right)^{a'_{r,1}} = x_{1,1}^{\ell-1}\otimes x_{2,1}^{2n} \otimes\cdots\otimes x_{k,1}^{2n}+\cdots \ne 0,
    \end{equation}
so that $\ell-1\in\mathcal{L}_{n,k}$ as well.

We are now ready to prove Theorem~\ref{thm: gap k ge 3}~\eqref{third}.

\begin{proof}[\textbf{Case (2): $h\le (g-1)(\lambda-1)$}]
   Note that $g\geq2$. Write $h=\alpha(\lambda-1)+\beta$, where $\alpha\in \{0,\ldots,g-2\}$ and $\beta\in \{0,\ldots,\lambda-1\}$. Working in the context of Subsection~\ref{subsec: prep} for $\ell-1=\max\mathcal{L}_{n,k}-1\in\mathcal{L}_{n,k}$, take the product in~\eqref{eq: for even k2} as the corresponding non-zero expression~\eqref{eq: assume}. Also, for $(r,i)\in\{2,\ldots,k\}\times\{2,\ldots,g\}$, take
   \[
a'_{r,i}:=\begin{cases}
1 & \text{ if $i\leq \alpha+1$ and $r\leq k-1$, or $i=\alpha+2$ and $r\leq 2\beta+1$;} \\
0 & \text{ otherwise.}
\end{cases}
   \]
Then
    \[
    x_{1,1}^{\ell-1}\ \prod_{i=2}^gx_{1,i}^{\sum_{r=2}^{k}a'_{r,i}}=x_{1,1}^{2n-2h}\left(\prod_{i=2}^{\alpha+1}x_{1,i}^{k-2} \right) x_{1,\alpha+2}^{2\beta}=x_{1,1}^{2n-2h+(k-2)\alpha+2\beta}=x_{1,1}^{2n}\ne 0,
    \]
    which yields the non-triviality of the analogue of~\eqref{elmonomio} in the context of~\eqref{eq: for even k2}. Therefore, the second inequality in~\eqref{eq: imp1} gives the best possible estimate 
    \begin{align*}\gap_k(SP^n(N_g))&\le 2nk-\sum_{\substack{2\le r\le k \\ 1\le i \le g}} a'_{r,i}=2nk- \sum_{r=2}^k a'_{r,1}-\sum_{\substack{2\le r\le k \\ 2\le i \le g}} a'_{r,i}\\ &=2nk-(2n(k-1)+\ell-1)-(\alpha(k-2)+2\beta)\\&=2n -(2n-2h)-(\alpha(k-2)+2\beta)=0.
    \end{align*}
\end{proof}

Lastly, we prove Theorem~\ref{thm: gap k ge 3}~\eqref{second} by establishing the equivalent Theorem~\ref{thm: zcl k ge 3}~\eqref{2nd}. 

\begin{proof}[\textbf{Case (3): $(g-1)(\lambda-1)<h\le (g-1)\lambda$}]
Our task is to establish the equality $\zcl_k(SP^n(N_g))=2nk-t$, where $t:=h-(g-1)(\lambda-1)$. Clearly $1\leq t\leq g-1$, so $g\geq2$. As in the proof of Case (1), we work in the context of Subsection~\ref{subsec: prep} for $\ell=\max\mathcal{L}_{n,k}$, taking the product in~\eqref{eq: for even k1} as the corresponding non-zero expression~\eqref{eq: assume}. Furthermore, for $(r,i)\in\{2,\ldots,k\}\times\{2,\ldots,g\}$, take
   \[
a_{r,i}:=\begin{cases}
1 & \text{ if $i\leq t$ or $r\le k-1$;} \\
0 & \text{ otherwise.}
\end{cases}
   \]
Then
\begin{align*}
x_{1,1}^{\ell}\ \prod_{i=2}^gx_{1,i}^{\sum_{r=2}^{k}a_{r,i}}&=x_{1,1}^{2n-2h+1}\left(\prod_{i=2}^tx_{1,i}^{k-1}\right)\left( \prod_{i=t+1}^gx_{1,i}^{k-2}\right)\\
&= x_{1,1}^{2n-2h}\left(\prod_{i=2}^gx_{1,i}^{k-2}\right)\left(\prod_{i=1}^tx_{1,i}\right) \\
&=x_{1,1}^{2n-2h+(k-2)(g-1)} \prod_{i=1}^tx_{1,i}\ne 0
\end{align*}
    because $n-h+(\lambda-1)(g-1)+t=n$. This yields the non-triviality of~\eqref{elmonomio}, so the first estimate in~\eqref{eq: imp2} gives
    \begin{align*}
    \zcl_k(SP^n(N_g))&\ge \zcl_k( P^{2n} )+ \sum_{\substack{2\le r\le k \\ 2\le i \le g}}a_{r,i} \\ &= 2nk-2h+1+(k-2)(g-1)+t-1 = 2nk-t.
    \end{align*}
The much subtler fact that this estimate is sharp  follows from Proposition~\ref{prop: zcl search reduced} and Lemma~\ref{prop: final} below, which implies that, in the current setting, any product $m\cdot e \in H^*(SP^n(N_g))^{\otimes k}$ with factors
\begin{align} 
m&=\prod_{r=2}^k\left(x_{1,1}+x_{r,1}\right)^{\alpha_{r,1}} \in H^*(P^{2n})^{\otimes k}, \ \text{ with $\alpha_{r,1}\geq0$},\ \text{ and} \nonumber \\
e&=\prod_{\substack{2\le r\le k\\ 2\le i \le g}}\left(x_{1,i}+x_{r,i}\right)^{\delta_{r,i}} \in H^*(SP^n(N_g))^{\otimes k}, \ \text{ with $\delta_{r,i}\in\{0,1\}$}, \label{eq: e0}   \end{align}
vanishes whenever $\deg(m\cdot e)>2nk-t$.
\end{proof}

The proof of the final ingredient, Lemma~\ref{prop: final} below, makes key use of the capacity function introduced in Definition~\ref{def: capacity}.

\begin{lemma}\label{prop: final}
    Let $n,g\geq1$, $d\geq0$, $k=2\lambda\ge 4$, and $\gap_k(P^{2n}):=2h-1\geq1$ such that $(g-1)(\lambda-1)<h\le(g-1)\lambda$. Let $t:=h-(g-1)(\lambda-1)$ and assume that
\begin{enumerate}
        \item $m\in H^*(P^{2n})^{\otimes k}$ is homogeneous with $\deg(m)=\zcl_k(P^{2n})-d$, and\hspace{.5mm}\footnote{We note that the equality $\deg(m)=\zcl_k(P^{2n})-d$ makes sense since the hypotheses of Lemma~\ref{prop: final} imply $d<\gap_k(P^{2n})$, whose verification is easy and hence omitted.}
        \item  $e\in H^*(SP^n(N_g))^{\otimes k}$ is as in~\eqref{eq: e0} with $\deg(e)>\gap_k(P^{2n})+d-t$.
    \end{enumerate}
    Then the product $m\cdot e$ vanishes in $H^*(SP^n(N_g))^{\otimes k}$.
\end{lemma}

We have observed that the hypotheses on $h$ imply $g\ge 2$. So, Lemma~\ref{prop: final} is unnecessary in Davis's study~\cite{Dav} of the $k$-th zero-divisors of $P^{2n}$. This is another indication of the fact that the exact computation of $\zcl_k(SP^n(N_g))$ for $g\ge 2$ is strictly more involved than that for $g=1$.

\begin{proof}[Proof of Lemma~\ref{prop: final}]
We can safely assume $m=x_{1,1}^{2n-p_1}\otimes\cdots\otimes x_{k,1}^{2n-p_k}$, with $p_r\ge 0$ for $1\le r\le k$, as well as $\deg(e)=\gap_k(P^{2n})+d-t+1$, which, by hypothesis, agrees with $d+h+(g-1)(\lambda-1)$. Then $p_1+\cdots+p_k=2h-1+d$, so that Proposition~\ref{prop: prop y from notes} gives
\[
c(m)\le\left\lfloor\frac{2h-1+d}{2}\right\rfloor=h+\left\lfloor\frac{d-1}{2}\right\rfloor.
\]
The vanishing of $m\cdot e$ then follows from an iterated application of Proposition~\ref{prop: prop z from notes} and the fact that
\[
\sum_{i=2}^g\left\lfloor \frac{\delta_i+1}{2} \right\rfloor>h+\left\lfloor\frac{d-1}{2}\right\rfloor,
\]
where $\delta_i:=\sum_{r=2}^k\delta_{r,i}$.  In turn, the previous inequality is a special case of the purely algebraic relation
\begin{equation}\label{alternativa}
\min\left\{ \,\sum_{i=2}^g\left\lfloor \frac{\varepsilon_i+1}{2} \right\rfloor\ \middle|\ 
\begin{matrix}
0\le\varepsilon_i\le k-1, &  \\ \displaystyle\sum_{i=2}^g\varepsilon_i=d+h+(g-1)(\lambda-1) &
\end{matrix}\hspace{-3mm}\right\}>h+\left\lfloor\frac{d-1}2\right\rfloor,
\end{equation}
whose validity is argued next.

First, note that $\delta_i\leq k-1=2\lambda-1$ for $2\leq i\leq g$, so 
\[
(g-1)(2\lambda-1)\geq\sum_{i=2}^g\delta_i=\deg(e)=d+h+(g-1)(\lambda-1),
\]
which yields the third inequality in
$(g-1)(\lambda-1)<h\leq d+h\le (g-1)\lambda$
(the other two hold by hypothesis). We can then write $d+h=(g-1)\lambda-j$ for some $j\in\{0,\ldots,g-2\}$, so that every summation $\sum\varepsilon_i$ in~\eqref{alternativa} equals $(g-1)(k-1)-j$. It follows that any tuple $\varepsilon:=(\varepsilon_2,\ldots,\varepsilon_g)$ having $j$ coordinates equal to $k-2$ and the rest equal to $k-1$, which is unique up to permutation, gives an instance of the set in~\eqref{alternativa} that also satisfies
\begin{align*}
\sum_{i=2}^g\left\lfloor \frac{\varepsilon_i+1}{2} \right\rfloor &= \sum_{i=2}^{j+1}\left\lfloor \frac{k-1}{2}\right\rfloor+\sum_{i=j+2}^g\left\lfloor \frac{k}{2}  \right\rfloor\\&=j(\lambda-1)
+(g-j-1)\lambda=\lambda(g-1)-j>h+\left\lfloor\frac{d-1}2\right\rfloor,
\end{align*}
where the last inequality holds because $h=(g-1)\lambda-j-d$ and $d>\left\lfloor\frac{d-1}2\right\rfloor$. The inequality in~\eqref{alternativa} will then follow once we argue that
\begin{equation}\label{serealiza}
\sum_{i=2}^g\left\lfloor\frac{\varepsilon'_i+1}2\right\rfloor\geq \sum_{i=2}^g\left\lfloor\frac{\varepsilon_i+1}2\right\rfloor
\end{equation}
holds for any other instance $\varepsilon':=(\varepsilon'_2,\ldots,\varepsilon'_g)$ of the set in~\eqref{alternativa}. For this, note that, just as the tuple $\varepsilon$ has \emph{exactly} $g-j-1$ coordinates (say the first ones) with the maximal possible value $k-1$, the tuple $\varepsilon'$ must have \emph{at least} $g-j-1$ coordinates (say again the first ones) with the maximal possible value $k-1$. We can ignore these first $g-j-1$ coordinates, as they represent no difference for both sides in~\eqref{serealiza}. From the remaining coordinates, we can likewise ignore those in $\varepsilon'$ together with the corresponding ones in $\varepsilon$ with value $k-2$. Let $R\subset\{2,\ldots,g\}$ denote the set of the remaining coordinates. All of these then have value $k-2$ in the case of $\varepsilon$, while value different from $k-2$ in the case of $\varepsilon'$. Let $R'\subseteq R$ (resp.~$R''\subseteq R$) consist of the coordinates where $\varepsilon'$ has value strictly smaller than $k-2$ (resp.~value equal to $k-1$). Since $\varepsilon_i+1=k-1$  is odd for $i\in R$, we see that
\begin{align}
i\in R'& \implies 0\leq\left\lfloor\frac{\varepsilon_i+1}{2}\right\rfloor-\left\lfloor\frac{\varepsilon'_i+1}{2}\right\rfloor=\left\lfloor\frac{\varepsilon_i-\varepsilon'_i}{2}\right\rfloor; \label{eq: for R'}
\\
i\in R''& \implies  1=\left\lfloor\frac{\varepsilon'_i+1}{2}\right\rfloor-\left\lfloor\frac{\varepsilon_i+1}{2}\right\rfloor, \text{ as well as } \varepsilon'_i-\varepsilon_i=1.\label{eq: for R''}
\end{align}
By definition of $\varepsilon'$ we have $\sum_{i=2}^g\varepsilon_i=\sum_{i=2}^g\varepsilon_i'$, which by construction yields
\begin{equation*}
    \sum_{i\in R'}\varepsilon_i+ \sum_{i\in R''}\varepsilon_i=\sum_{i\in R'}\varepsilon_i'+\sum_{i\in R''}\varepsilon_i',
\end{equation*}
so that
\[
\sum_{i\in R'}\left(\varepsilon_i-\varepsilon_i'\right)= \sum_{i\in R''}\left(\varepsilon_i'-\varepsilon_i\right)=|R''|,
\]
where the last equality uses~\eqref{eq: for R''}. From this equality,~\eqref{eq: for R'}, and~\eqref{eq: for R''}, we get that
\begin{align*}
\sum_{i=2}^g \left(\left\lfloor\frac{\varepsilon_i'+1}{2}\right\rfloor-\left\lfloor\frac{\varepsilon_i+1}{2}\right\rfloor\right) & = \sum_{i\in R'} \left(\left\lfloor\frac{\varepsilon_i'+1}{2}\right\rfloor-\left\lfloor\frac{\varepsilon_i+1}{2}\right\rfloor\right) + \sum_{i\in R''} \left(\left\lfloor\frac{\varepsilon_i'+1}{2}\right\rfloor-\left\lfloor\frac{\varepsilon_i+1}{2}\right\rfloor\right) 
\\ 
& =\sum_{i\in R'}\left(-\left\lfloor\frac{\varepsilon_i-\varepsilon_i'}{2}\right\rfloor\right)+|R''| 
\\ 
&= \sum_{i\in R'}\left(-\left\lfloor\frac{\varepsilon_i-\varepsilon_i'}{2}\right\rfloor\right)+\sum_{i\in R'}\left(\varepsilon_i-\varepsilon_i'\right) 
\\ 
&= \sum_{i\in R'}\left(\left(\varepsilon_i-\varepsilon_i'\right)-\left\lfloor\frac{\varepsilon_i-\varepsilon_i'}{2}\right\rfloor\right)\ge 0.
\end{align*}
This verifies~\eqref{serealiza}, thereby completing the proof. 
\end{proof}


\begin{thebibliography}{references}
\bibitem[BGRT14]{BGRT} I.~Basabe, J.~Gonz\'alez, Y.~B.~Rudyak, D.~Tamaki, Higher topological complexity and its symmetrization. \textit{Algebr. Geom. Topol.} \textbf{14} (2014), no.~4, 2103--2124.

\bibitem[BGZ02]{BGZ} P.~Blagojevi\'c, V. Gruji\'c, R. Zivaljevi\'c, Symmetric products of surfaces; a unifying theme for topology and physics. In \emph{Proceedings of Summer School in Modern Mathematical Physics}, ed. B. Dragovic \emph{et al.}, Sveske Fiz. Nauka, Ser. A, Conf. 3, Institute of Physics, Belgrade, 2002, pp.~466--491.

\bibitem[CAC+26]{CAC} N.~Cadavid-Aguilar, D.~Cohen, J.~Gonz\'alez, S.~Hughes, L. Vandembroucq, On the higher topological complexity of manifolds with abelian fundamental group. \emph{J. Appl. Comput. Topol.} \textbf{10} (2026), no.~1, 4, pp.~26.

\bibitem[CAG+18]{CAG+} N.~Cadavid-Aguilar, J.~Gonz\'alez, D.~Guti\'errez, A.~Guzm\'an-S\'aenz, A.~Lara, Sequential motion planning algorithms in real projective spaces: an approach to their immersion dimension. \emph{Forum Math.} \textbf{30} (2018), no.~2, 397--417.

\bibitem[CV17]{CV1} D.~C.~Cohen, L.~Vandembroucq, Topological complexity of the Klein bottle. \emph{J. Appl. Comput. Topol.} \textbf{1} (2017), no.~2, 199--213.

\bibitem[CV21]{CV2} D.~C.~Cohen, L.~Vandembroucq, On the topological complexity of manifolds with abelian fundamental group. \emph{Forum Math.} \textbf{33} (2021), no.~6, 1403--1421.

\bibitem[CLOT03]{CLOT} O.~Cornea, G.~Lupton, J.~Oprea, D.~Tanr\'e, \emph{Lusternik--Schnirelmann Category}. Math. Surveys Monogr., 103, AMS, Providence, 2003, xviii+330~pp.

\bibitem[Dav18]{Dav} D.~M.~Davis, A lower bound for higher topological complexity of real projective space. \emph{J. Pure Appl. Algebra} \textbf{222} (2018), no.~10, 2881--2887.

\bibitem[DCDJ]{DCDJ} L.~F.~Di~Cerbo, A.~Dranishnikov, E.~Jauhari, Curvature, macroscopic dimensions, and symmetric products of surfaces. Preprint, arXiv:2503.01779 [math.GT] (2025), pp. 41.

\bibitem[Dra14]{Dr1} A.~Dranishnikov, Topological complexity of wedges and covering maps. \emph{Proc. Amer. Math. Soc.} \textbf{142} (2014), no.~12, 4365--4376.

\bibitem[Dra16]{Dr2} A.~Dranishnikov, The topological complexity and the homotopy cofiber of the diagonal map for non-orientable surfaces. \emph{Proc. Amer. Math. Soc.} \textbf{144} (2016), no.~11, 4999--5014.

\bibitem[Dra17]{Dr3} A.~Dranishnikov, On topological complexity of non-orientable surfaces. \emph{Topology Appl.} \textbf{232} (2017), 61--69.

\bibitem[Far03]{Far} M.~Farber, Topological complexity of motion planning. \emph{Discrete Comput. Geom.} \textbf{29} (2003), no.~2, 211--221.

\bibitem[Far06]{Far2} M.~Farber, Topology of robot motion planning. In \textit{Morse theoretic methods in nonlinear analysis and in symplectic topology}, ed. P.~Biran \textit{et al.}, NATO Sci. Ser. II Math. Phys. Chem., 217, Springer, Dordrecht, 2006, pp. 185--230.


\bibitem[FKS20]{FKS} M.~Farber, D.~Kishimoto, D.~Stanley, Generating functions and topological complexity. \textit{Topology Appl.} \textbf{278} (2020), 107235, pp. 5.

\bibitem[FO19]{FO} M.~Farber, J.~Oprea, Higher topological complexity of aspherical spaces. \textit{Topology Appl.} \textbf{258} (2019), 142--160.

\bibitem[FTY03]{FTY} M.~Farber, S.~Tabachnikov, S.~Yuzvinsky, Topological robotics: motion planning in projective spaces. \emph{Intl. Math. Res. Not.} (2003), no.~34, 1853--1870.

\bibitem[GCV13]{GV} J.~M.~Garc\'ia-Calcines, L.~Vandembroucq, Topological complexity and the homotopy cofibre of the diagonal map. \emph{Math. Z.} \textbf{274} (2013), no.~1-2, 145--165.

\bibitem[GGGL16]{GGGL} J. Gonz\'alez, B. Guti\'errez, D. Guti\'errez, A. Lara, Motion planning in real flag manifolds. \emph{Homology Homotopy Appl.} \textbf{18} (2016), no.~2, 359--375.

\bibitem[GGV18]{GGL} J. Gonz\'alez, M.~Grant, L.~Vandembroucq, Hopf invariants, topological complexity, and LS-category of the cofiber of the diagonal map for two-cell complexes. In \emph{Topological complexity and related topics}, Contemp. Math., 702, AMS, Providence, 2018, pp.~133--150.

\bibitem[HL22]{HL} S.~Hughes, K.~Li, Higher topological complexity of hyperbolic groups. \emph{J. Appl. Comput. Topol.} \textbf{6} (2022), no.~3, 323--329.


\bibitem[Hat02]{Ha} A.~Hatcher, {\em Algebraic Topology}. Cambridge University Press, Cambridge, 2002, xii+544~pp.

\bibitem[IS10]{IS} N.~Iwase, M.~Sakai, Topological complexity is a fibrewise L-S category. \emph{Topology Appl.} \textbf{157} (2010), no.~1, 10--21

\bibitem[Jau25]{Ja} E.~Jauhari, LS-category and sequential topological complexity of symmetric products. \emph{J. Appl. Comput. Topol.} \textbf{9} (2025), no.~4, 25, pp. 31.

\bibitem[KS06]{KS} S.~Kallel, P.~Salvatore, Symmetric products of two dimensional complexes. In {\em Recent developments in algebraic topology}, ed. A.~\'Adem {\em et al.}, Contemp. Math., 407, AMS, Providence, 2006, pp. 147--161.

\bibitem[KT13]{KT} S.~Kallel, W.~Taamallah, The geometry and fundamental group of permutation products and fat diagonals. \emph{Canad. J. Math.} \textbf{65} (2013), no.~3, 575--599.

\bibitem[Mac62]{Mac} I.~G.~Macdonald, Symmetric products of an algebraic curve. \emph{Topology} \textbf{1} (1962), 319--343.

\bibitem[Rud10]{Ru} Y.~B.~Rudyak, On higher analogs of topological complexity. \textit{Topology Appl.} \textbf{157} (2010), no. 5, 916--920. 
\end{thebibliography}
\end{document}